\documentclass[10pt]{amsart}
\usepackage{amsmath}
\usepackage{amsthm}
\usepackage{amsfonts}
\usepackage{amssymb}
\usepackage[shortlabels]{enumitem}
\usepackage[pagebackref,hypertexnames=false, colorlinks, citecolor=blue, linkcolor=red, pdfstartview=FitB]{hyperref}
\usepackage[latin1]{inputenc}

\newcommand{\cA}{\mathcal A}
\newcommand{\cH}{\mathcal H}
\newcommand{\cI}{\mathcal I}
\newcommand{\bC}{\mathbb C}
\newcommand{\bB}{\mathbb B}
\newcommand{\cB}{\mathcal B}
\newcommand{\cK}{\mathcal K}
\newcommand{\cE}{\mathcal E}
\newcommand{\bD}{\mathbb D}
\newcommand{\bT}{\mathbb T}
\newcommand{\cS}{\mathcal S}
\newcommand{\bN}{\mathbb N}
\newcommand{\cL}{\mathcal L}
\newcommand{\cM}{\mathcal M}
\DeclareMathOperator{\Mult}{Mult}

\newcommand{\ol}{\overline }
\renewcommand{\phi}{\varphi}

\newtheorem{lemma}{Lemma}[section]
\newtheorem{theorem}[lemma]{Theorem}
\newtheorem{proposition}[lemma]{Proposition}
\newtheorem{corollary}[lemma]{Corollary}
\theoremstyle{definition}

\newtheorem{remark}[lemma]{Remark}
\newtheorem{question}{Question}
\newtheorem{example}[lemma]{Example}

\renewcommand{\MR}[1]{} 

\begin{document}

\author[R. Clou\^atre]{Rapha\"el Clou\^atre}
\address{Department of Mathematics, University of Manitoba, Winnipeg, Manitoba, Canada R3T 2N2}
\email{raphael.clouatre@umanitoba.ca\vspace{-2ex}}
\thanks{The first author was partially supported by an FQRNT postdoctoral fellowship and an NSERC Discovery Grant.}

\author[M. Hartz]{Michael Hartz}
\address{Department of Mathematics, Washington University in St. Louis, One Brookings Drive,
St. Louis, MO 63130, USA}
\email{mphartz@wustl.edu}
\thanks{The second author was partially supported by an Ontario Trillium Scholarship and a Feodor Lynen Fellowship.}

\title[Multiplier algebras of complete Nevanlinna-Pick spaces]{Multiplier algebras of complete Nevanlinna-Pick spaces: dilations, boundary representations and hyperrigidity}
\date{\today}
\subjclass[2010]{Primary 47L55; Secondary 46E22, 47A13}
\keywords{Multiplier algebra, Nevanlinna-Pick space, dilation, coextension, boundary representation,
hyperrigidity, homogeneous ideals}

\begin{abstract}
We study reproducing kernel Hilbert spaces on the unit ball with the complete Nevanlinna-Pick property through the representation theory of their algebras of multipliers. We give a complete description of the representations in terms of the reproducing kernels. While representations always dilate to $*$-representations
  of the ambient $C^*$-algebra, we show that in our setting we automatically obtain coextensions.
  In fact, we show that in many cases, this phenomenon characterizes
  the complete Nevanlinna-Pick property.
  We also deduce operator theoretic dilation results which are in the spirit of
  work of Agler and several other authors.
  Moreover, we identify all boundary representations,
  compute the $C^*$-envelopes and determine hyperrigidity for certain analogues of the disc algebra.
  Finally, we extend these results to spaces of functions
  on homogeneous subvarieties of the ball.
\end{abstract}

\maketitle
 \section{Introduction}\label{S:intro}

In this work, we study representations of multiplier algebras of reproducing kernel
Hilbert spaces with the complete Nevanlinna-Pick property.
Our motivating example is the classical disc algebra $A(\bD)$,
which consists of those analytic functions on the open unit disc $\bD\subset \bC$ which extend to continuous functions on $\ol{\bD}$.
This algebra plays a pivotal role in the theory of contractions on Hilbert space,
see \cite{nagy2010} for a classical treatment. To make this role explicit, we must consider $A(\bD)$ in relation to the Hardy space $H^2(\bD)$, which is the Hilbert space of analytic functions on $\bD$ whose Taylor coefficients at the origin
are square summable. Alternatively, the Hardy space may be described as the reproducing kernel
Hilbert space on the unit disc with kernel
\[
k(z,w)=\frac{1}{1-z \ol{w}} \quad (z,w\in\bD).
\]
Then, for $\phi\in A(\bD)$ we denote by $M_\phi$ the corresponding multiplication operator on $H^2(\bD)$. This identification allows us to regard $A(\bD)$ as an algebra of bounded linear operators on $H^2(\bD)$.

The (matrix-valued version of) von Neumann's inequality (see, for example, \cite[Corollary 3.12]{paulsen2002})
shows that unital (completely)
contractive representations of $A(\bD)$
are in one-to-one correspondence with contractions on a Hilbert space. Specifically,
every contraction $T \in B(\cH)$ gives rise to a unital completely contractive
representation
\[
A(\bD)\to B(\cH), \quad
  p \mapsto p(T) \quad (p \in \bC[z]).
\]
Abstract dilation theory in the form of Arveson's dilation theorem \cite{arveson1969}  (see also
\cite[Corollary 7.7]{paulsen2002}) shows that this
representation dilates to a $*$-representation $\pi$ of $C^*(A(\bD)) \subset B(H^2(\bD))$, in the sense that
\[
  p(T)=P_{\cH}\pi(M_p)|_{\cH} \quad (p \in \bC[z]).
\]
The $C^*$-algebra $C^*(A(\bD))$ is known as the Toeplitz algebra, whose representations are well understood.
Indeed, $M_z$ is the unilateral shift, hence $\pi(M_z)$ is an isometry.
The Wold decomposition therefore implies that
$\pi(M_z)$ is unitarily equivalent to $M_z^{(\alpha)}\oplus U$ for some cardinal $\alpha$ and some unitary operator $U$.
In particular, we see that every contraction $T$ dilates to an operator of the form $M_z^{(\alpha)}\oplus U$.

However, a stronger statement is true.
Namely, one version of Sz.-Nagy's dilation theorem \cite[Theorem 4.1]{nagy2010}, combined with the Wold decomposition,
asserts that every contraction $T$ does not merely dilate, but in fact coextends
to an operator of the form $M_z^{(\alpha)} \oplus U$. That is, the dilation has the additional property that $\cH$ is invariant
under $(M_z^{(\alpha)} \oplus U)^*$. In representation theoretic terms, this means
that every unital completely contractive representation $\rho: A(\bD) \to B(\cH)$ coextends
to a $*$-representation $\pi$ of $C^*(A(\bD))$, in the sense that
\begin{equation*}
  \rho(M_p)^* = \pi(M_p)^* \big|_{\cH} \quad (p \in \bC[z]). 
\end{equation*}
Thus, when viewing $A(\bD)$ as a subalgebra of $B(H^2(\bD))$, we conclude that its representations have a finer property than what is usually expected from general dilation theory.

The multivariate counterpart to this discussion takes place in the so-called Drury-Arveson space $H^2_d$ on the open unit ball $\bB_d\subset \bC^d$ \cite{drury1978,muller1993,arveson1998}, see also the survey \cite{Shalit13}.
This is the reproducing kernel Hilbert space on $\bB_d$ with kernel
\[
k(z,w)=\frac{1}{1-\langle z,w\rangle_{\bC^d}} \quad (z,w\in \bB_d).
\]
Let $\cA_d$ be the norm closed subalgebra of $B(H^2_d)$ generated by the multiplication operators $M_{z_1},\ldots,M_{z_d}$.
This algebra plays the role of the disc algebra in the study of \emph{commuting row contractions},
that is, tuples $T = (T_1,\ldots,T_d)$ of commuting operators such that $\sum_{j=1}^d T_j T_j^* \le I$.
As before, Arveson's dilation theorem shows that every unital
completely contractive representation of $\cA_d$ dilates to a $*$-representation $\pi$ of
the Toeplitz algebra $C^*(\cA_d) \subset B(H^2_d)$ (and one can then use the known representation theory
of the Toeplitz algebra to deduce an operator theoretic dilation result, see \cite[Proposition 8.3]{arveson1998}).
Once again, however, something stronger is true. A dilation theorem due to M\"uller-Vasilescu \cite{muller1993}
and Arveson \cite{arveson1998} for commuting row contractions implies that unital completely contractive representations
of $\cA_d$ in fact coextend to $*$-representations of $C^*(\cA_d)$.

Motivated by these examples, we ask:
\begin{question}\label{Q:intro1}
  Let $\cA$ be a unital algebra of multipliers of a reproducing kernel Hilbert space, and let $\rho$ be a unital
  completely contractive representation of $\cA$. Does $\rho$ coextend to a $*$-representation
  of $C^*(\cA)$?
\end{question}

The discussion above shows that this question always has a positive answer for $A(\bD) \subset B(H^2(\bD))$
and, more generally, for $\cA_d \subset B(H^2_d)$.

We saw above that the existence of an isometric dilation for a single contraction $T\in B(\cH)$
could be deduced from the inclusions
\[
A(\bD)\subset C^*(A(\bD)) \subset B(H^2(\bD)).
\]
A different version of Sz.-Nagy's dilation \cite[Theorem I.4.2]{nagy2010} can be obtained by embedding $A(\bD)$ inside a different $C^*$-algebra.
By the maximum modulus principle, if we denote by $\bT$ the unit circle, then we may regard $A( \bD)$ as a subalgebra of $C(\bT)$, via restriction to $\bT$.
In fact, the $C^*$-algebra $C(\bT)$ is the smallest $C^*$-algebra which contains
a copy of $A(\bD)$ in a precise sense:
it is the $C^*$-envelope of $A(\bD)$.
An application of Arveson's dilation theorem in this setting yields a $*$-representation $\pi$ of $C(\bT)$ such that
\[
  p(T)=P_{\cH}\pi(p)|_{\cH} \quad (p \in \bC[z]).
\]
Since the identity function $z$ is unitary in $C(\bT)$, we thus obtain a unitary dilation of $T$.
It is then natural to wonder whether a similar approach is possible for commuting row contractions
and the algebra $\cA_d$.
Unfortunately this is not the case, as the $C^*$-algebra $C^*(\cA_d)$ considered above is the
$C^*$-envelope of $\cA_d$. In fact, a stronger statement is true: the identity representation
is a boundary representation for $\cA_d$ \cite[Lemma 7.13]{arveson1998} (we will review
these notions in Subsection \ref{S:prelimuep}).

This motivates our second main question:

\begin{question} \label{Q:intro2}
Given an algebra of multipliers $\cA$, what is its $C^*$-envelope? What are its boundary representations?
 \end{question}

In particular, we will study Questions \ref{Q:intro1} and \ref{Q:intro2}
for algebras of multipliers of reproducing kernel Hilbert spaces with the complete Nevanlinna-Pick property.
These spaces mirror some of the fine structure of the classical Hardy space. We will
recall the precise definition in Subsection \ref{ss:kernels}. For now, we only mention that typical examples
are the Hardy space itself, the Drury-Arveson space and also the Dirichlet space.

In the context of Question \ref{Q:intro1}, we will show in Sections \ref{S:incl} and \ref{S:coextue}
that the ``automatic coextension property'',
which we witnessed for $A(\bD)$ and $\cA_d$ above, holds for many algebras of multipliers
of complete Nevanlinna-Pick spaces.
In fact, we will see in Section \ref{sec:NP_nec} that in many cases of interest,
this phenomenon characterizes those spaces with the complete Nevanlinna-Pick property.
As for Question \ref{Q:intro2}, we show in Section \ref{S:AH} that
the fact that the $C^*$-envelope of $\cA_d$ is $C^*(\cA_d) \subset B(H^2_d)$ is the typical behaviour
for many complete Nevanlinna-Pick spaces, and that
the Hardy space is an exception in that regard.

We now proceed to state our main results more precisely and to explain how the paper is organized.
Section \ref{S:prelimRKHS} introduces the necessary background material and preliminary results about reproducing kernel Hilbert spaces.
In particular, we recall the definition of unitarily invariant complete Nevanlinna-Pick spaces on the unit ball.
To such a space $\cH$, we associate an algebra of multipliers $A(\cH)$, which is the norm closure
of the polynomial multipliers and serves as a
generalization of the algebras $A(\bD)$ and $\cA_d$ mentioned above.
Most of our results deal with these algebras. In Section \ref{sec:prelim_dilations},
we recall the necessary background material about dilations.

In Section \ref{S:incl}, we study a very concrete type of representations,
namely inclusions of multiplier algebras of two reproducing kernel Hilbert spaces on the same set.
In particular, we show the following result (Theorem \ref{T:equiv}).

\begin{theorem}\label{T:main1}
Let $\cH_1$ and $\cH_2$ be two reproducing kernel Hilbert spaces on the same set $X$
with reproducing kernels $k_1,k_2$, respectively. Suppose that $\cH_1$
is an irreducible complete Nevanlinna-Pick space. Then the following assertions are equivalent.
\begin{enumerate}[label=\normalfont{(\roman*)}]
  \item The inclusion $\Mult(\cH_1) \hookrightarrow \Mult(\cH_2)$ is completely contractive.
  \item The inclusion $\Mult(\cH_1)\hookrightarrow \Mult(\cH_2)$ coextends
    to a $*$-representation of $B(\cH_1)$.
  \item $\Mult(\cH_1) \subset \Mult(\cH_2)$ and there exists an auxiliary
    Hilbert space $\cE$ and an isometry $V: \cH_2 \to \cH_1 \otimes \cE$ such that
    $V^*(M_{\phi}^{\cH_1}\otimes I_{\cE})=M_{\phi}^{\cH_2}V^*$
    for every $\varphi \in \Mult(\cH_1)$.
  \item $k_2/k_1$ is positive semi-definite on $X$.
\end{enumerate}
\end{theorem}

This result is in fact a combination of previously known theorems of
Douglas-Misra-Sarkar \cite{douglas2012dilation},
Kumari-Sarkar-Sarkar \cite{KSS+16} and a result from \cite{hartz2015isom},
but it appears to be
new in this form. Moreover, it shows that Question \ref{Q:intro1} has a positive answer for these particular
representations.
In fact, we show that the validity of the implication (i) $\Rightarrow$ (iv) characterizes those
reproducing kernel Hilbert spaces with the complete Nevanlinna-Pick property (Theorem \ref{T:NPchar}),
which gives the first indication that this property indeed plays a key role.

In Section \ref{S:coextue}, we consider representations of the algebras $A(\cH)$ associated to
unitarily invariant complete Nevanlinna-Pick spaces on the unit ball.
We describe all unital completely contractive representations of these algebras
and give a positive answer to Question \ref{Q:intro1} in this case.
To achieve this, we will make use of a general machinery for studying coextensions which
has been developed over the years by Agler \cite{agler1982}, Agler-McCarthy \cite{AM00a} and Ambrozie-Engli\v{s}-M\"uller \cite{ambrozie2002}.
The main result of that section is summarized in the following theorem, which is a combination of
Theorem
\ref{T:A_H_representation} and Corollary \ref{cor:automatic_coextension}.

\begin{theorem}\label{T:main3}
  Let $\cH$ be a unitarily invariant complete Nevanlinna-Pick space on the unit ball
  with reproducing kernel $k$. Let $T=(T_1,\ldots,T_d)$ be a commuting tuple of operators on a Hilbert space $\cE$.
Then, the following assertions are equivalent.
      
\begin{enumerate}[label=\normalfont{(\roman*)}]
\item The tuple $T$ satisfies $1/k(T,T^*)\geq 0$.

\item  There exists a unital completely contractive homomorphism 
\[
\rho: A(\cH) \to B(\cE)
\]
such that $\rho(M_z) = T$. 
 \end{enumerate}
 In this case, the map $\rho$ is necessarily unique and coextends to a $*$-representation of $C^*(A(\cH))$.
 In particular, every unital completely contractive representation of $A(\cH)$ coextends
 to a $*$-representation of $C^*(A(\cH))$.
\end{theorem}
The precise definition of $1/k(T,T^*)$ can be found in Section \ref{S:coextue}.
For now, we simply mention that if $\cH = H^2_d$, then 
\[
1/k(T,T^*) = I - \sum_{j=1}^d T_j T_j^*
\]
so $1/k(T,T^*) \ge 0$ if and only if $T$ is a row contraction.

As a consequence, we establish the following dilation
theorem (Theorem \ref{T:1/k_dilation}),
which is very much in the spirit of \cite{agler1982,AM00a,ambrozie2002} and in addition
contains an automatic coextension statement.
It extends the M\"uller-Vasilescu and Arveson dilation theorem \cite{muller1993,arveson1998}.
Recall that a spherical unitary is a tuple
of commuting normal operators whose joint spectrum is contained in the unit sphere.

\begin{theorem}
  Let $\cH$ be a unitarily invariant complete Nevanlinna-Pick space on $\bB_d$
  with reproducing kernel
\[
k(z,w)=\sum_{n=0}^\infty a_n \langle z,w \rangle^n,
\]
where $a_0 = 1$ and $a_n > 0$ for all $n \in \mathbb N$. Suppose that $\lim_{n \to \infty} a_n / a_{n+1} = 1$.
  Let $T = (T_1,\ldots,T_d)$ be a commuting
  tuple of operators on a Hilbert space. Denote by $M_z = (M_{z_1},\ldots,M_{z_d})$ the tuple of operators of multiplication by the coordinate functions. Then the following assertions
  are equivalent.
  \begin{enumerate}[label=\normalfont{(\roman*)}]
\item The tuple $T$ satisfies $1/k(T,T^*) \ge 0$.
    \item
  The tuple $T$ coextends
  to $M_z^{(\kappa)} \oplus U$ for some cardinal $\kappa$
  and a spherical unitary $U$.
\item
  The tuple $T$ dilates
  to $M_z^{(\kappa)} \oplus U$ for some cardinal $\kappa$
  and a spherical unitary $U$.
  \end{enumerate}
\end{theorem}

The regularity condition on the coefficients $a_n$ is discussed in Subsection \ref{ss:regular}. It is satisfied
for many spaces of interest including the Hardy space, the Drury-Arveson space and the Dirichlet space.

In Section \ref{S:AH}, we study Question \ref{Q:intro2} for the algebras $A(\cH)$. More precisely, we
identify all the boundary representations (Theorem \ref{T:bdryan}), compute
the $C^*$-envelope (Corollary \ref{C:envelope}), and determine
hyperrigidity (Theorem \ref{T:hyperrigid}) for the algebras $A(\cH)$, where $\cH$ is a unitarily invariant
complete Nevanlinna-Pick space which satisfies an additional regularity assumption.
(We will review these notions in Subsection \ref{S:prelimuep}.)
These results make use of earlier results by Guo-Hu-Xu \cite{GHX04} and can be summarized as follows.

\begin{theorem}\label{T:main4}
Let $\cH$ be a unitarily invariant complete Nevanlinna-Pick space on the unit ball with reproducing kernel 
\[
k(z,w)=\sum_{n=0}^\infty a_n \langle z,w \rangle^n,
\]
where $a_0 = 1$ and $a_n > 0$ for all $n \in \mathbb N$. Suppose that $\lim_{n \to \infty} a_n / a_{n+1} = 1$ and that $\cH$ is not the Hardy space $H^2(\bD)$.
\begin{enumerate}[label=\normalfont{(\alph*)}]
  \item If $\sum_{n=0}^\infty a_n = \infty$, 
then the boundary representations for $A(\cH)$ are precisely the characters of evaluation at a point on the sphere, along with the identity representation of $C^*(A(\cH))$. Furthermore, the algebra $A(\cH)$ is hyperrigid.
    \item  If $\sum_{n=0}^\infty a_n < \infty$, then
the identity representation of $C^*(A(\cH))$
      is the only boundary representation  for $A(\cH)$. In particular, the algebra $A(\cH)$ is not hyperrigid.
  \end{enumerate}
  In particular, the $C^*$-envelope of $A(\cH)$ is $C^*(A(\cH)) \subset B(\cH)$ in each case.
\end{theorem}
It is worth pointing out that all the $C^*$-algebras $C^*(A(\cH))$ are in fact unitarily equivalent
(Proposition \ref{P:toeplitzisom}). We now briefly address the situation where $\cH$ is the Hardy space. In that case, the algebra $A(\cH)$ coincides with the disc algebra $A(\bD)$, the boundary representations of which are well known to be the characters of evaluation at a point in $\bT$. In particular, the $C^*$-envelope of $A(\bD)$ is $C(\bT)$, and $A(\bD)$ is not hyperrigid inside $C^*(A(\bD)) \subset B(H^2(\bD))$.

In Sections \ref{S:ess_norm} and \ref{S:ideals}, we extend the results mentioned above to quotients
of $A(\cH)$ by homogeneous polynomial ideals. In this setting, Arveson's famous essential
normality conjecture plays an important role,
and we review some relevant background on this topic in Section \ref{S:ess_norm}.
Section \ref{S:ideals} contains the desired generalizations.
The results on boundary representations and hyperrigidity therein are related to 
some of the recent contributions of Kennedy and Shalit \cite{KS2015} on the Drury-Arveson space.
Moreover, they fit in well with the recent interest in operator algebras
associated to geometric varieties \cite{DRS2011,DRS2015,hartz2015isom}.

In the final Section \ref{sec:NP_nec}, we show that the automatic coextension property observed in Theorem \ref{T:main3}
is in fact equivalent to the complete Nevanlinna-Pick property for a familiar class
of reproducing kernel Hilbert spaces on the unit ball. The main result
of this section is the following theorem (Theorem \ref{thm:coext_char_regular}),
whose proof uses some of the machinery
on homogeneous ideals developed in Section \ref{S:ideals}.

\begin{theorem}
  Let $\cH$ be a unitarily invariant space on the unit ball
  with reproducing kernel
\[
k(z,w)=\sum_{n=0}^\infty a_n \langle z,w \rangle^n,
\]
where $a_0 = 1$ and $a_n > 0$ for all $n \in \mathbb N$. Suppose that $\lim_{n \to \infty} a_n / a_{n+1} = 1$.
  If every unital completely
  contractive representation of $A(\cH)$ coextends to a $*$-representation of $C^*(A(\cH))$, then
  $\cH$ is a complete Nevanlinna-Pick space.
\end{theorem}

\textbf{Acknowledgements.}
This project was initiated during the workshop ``Hilbert Modules and Complex Geometry'' at Mathematisches
Forschungsinstitut Oberwolfach in April 2014. The authors are grateful to the organizers of the workshop
and to the institute. The second-named author would like to thank John McCarthy for valuable discussions.

\section{Preliminaries on reproducing kernel Hilbert spaces}\label{S:prelimRKHS}

\subsection{Kernels and Nevanlinna-Pick spaces}
\label{ss:kernels}
Let $\cH$ be a reproducing kernel Hilbert space on a set $X$. Background material on reproducing kernel Hilbert spaces can be found in \cite{PR16}
and \cite{AM2002}. Let $k$ denote the reproducing kernel of $X$, which is positive semi-definite in the sense
that for any finite set of points $x_1,\ldots,x_n \in X$, the $n \times n$ matrix $[k(x_i,x_j)]_{i,j=1}^n$
is positive semi-definite.
The central object of our investigations is not the space $\cH$ itself as much as its \emph{multiplier algebra}
\begin{equation*}
  \Mult(\cH) = \{\varphi: X \to \bC: \varphi f \in \cH \text{ for all } f \in \cH \}.
\end{equation*}
An application of the closed graph theorem shows that for $\varphi \in \Mult(\cH)$,
the associated multiplication operator $M^{\cH}_\varphi$
is bounded on $\cH$.
Unless otherwise noted, we will always assume that $\cH$ has no common zeros, or equivalently,
that $k(w,w) \neq 0$ for all $w \in X$.
In this case, the assignment
$\varphi \mapsto M^{\cH}_\varphi$ is injective, so we may regard $\Mult(\cH)$ as a subalgebra
of $B(\cH)$ in this way.
When the space $\cH$ is clear from context, we simply write $M_\varphi$ instead of $M_\varphi^{\cH}$.

Most of the reproducing kernel Hilbert spaces we are interested in will be assumed to be \emph{irreducible} in the sense that $k(z,w) \neq 0$ for all $z,w \in X$, and $k(\cdot,w)$ and $k(\cdot,v)$ are linearly independent
if $v \neq w$ (see \cite[Section 7.1]{AM2002}). Moreover,
it will sometimes be convenient to assume that the reproducing kernel $k$
is \emph{normalized at $x_0 \in X$}, which means that $k(x,x_0) = 1$ for all $x \in X$.
A reproducing kernel Hilbert space $\cH$ is said to be a \emph{complete Nevanlinna-Pick space}
 if for all positive integers $r,n$, every collection of points
$z_1,\ldots,z_n$ in $X$ and every choice of $r\times r$ complex matrices $W_1,\ldots,W_n$, the non-negativity of the block matrix
\begin{equation*}
  \Big[ k(z_i,z_j) (I - W_i W_j^*) \Big]_{i,j=1}^n
\end{equation*}
is sufficient for the existence of a matrix-valued multiplier $\Phi$ \cite[Section 2.8]{AM2002}
with $\|\Phi\|\leq 1$
 such that
\begin{equation*}
  \Phi(z_i) = W_i, \quad i=1,\ldots,n.
\end{equation*}
Classical examples of complete Nevanlinna-Pick spaces include the Hardy space $H^2(\bD)$ and the Drury-Arveson space $H^2_d$ mentioned in the introduction. 
In general, irreducible complete Nevanlinna-Pick spaces are characterized in terms
of positivity of certain matrices by a
theorem of McCullough and Quiggin \cite{quiggin1993,mccullough1994}. A related characterization
is due to Agler and McCarthy \cite{AM2000}.
We refer the reader to \cite[Chapter 7]{AM2002} for a comprehensive account on the characterization
of complete Nevanlinna-Pick spaces.

\subsection{Unitarily invariant spaces}\label{SS:uispaces}
Although we will spend some time working at the level of generality described above, some of our finer results are obtained for a more concrete subclass of reproducing kernel Hilbert spaces, which we now introduce. Let $\bB_d\subset \bC^d$ denote the open unit ball. A \emph{unitarily invariant space} on $\bB_d$ is a reproducing kernel Hilbert space with a kernel of the form
 \begin{equation*}
   k(z,w) = \sum_{n=0}^\infty a_n \langle z,w \rangle^n
\end{equation*}
for some sequence of strictly positive coefficients $\{a_n\}_n$ such that $a_0 = 1$.
Since $a_0  =1$, the kernel $k$ is normalized at $0$.
It is well known that the monomials  $\{z^\alpha\}_{\alpha\in \bN^d\setminus \{0\}}$
form an orthogonal basis for $\cH$ and that
\[
\|z^\alpha\|^2_{\cH}=\frac{1}{a_{|\alpha|}}\frac{\alpha!}{|\alpha|!}.
\]
We use the usual multi-index notation: given $\alpha\in \bN^d\setminus \{0\}$ we have
\[
\alpha=(\alpha_1,\ldots,\alpha_d), \quad |\alpha|=\alpha_1+\ldots+\alpha_d, \quad \alpha!=\alpha_1!\cdots \alpha_d!
\]
and
\[
z^\alpha=z_1^{\alpha_1}\cdots z_d^{\alpha_d}.
\]
See \cite[Section 4]{greene2002} for more detail. 
Here we also feel compelled to point out that in view of \cite{AM2000}, the case where $d$ is an infinite cardinal number is important, however in this paper we will always assume that $d$ is a finite positive integer.

By a \emph{unitarily invariant complete Nevanlinna-Pick space} on the ball, we mean
a unitarily invariant space $\cH$ on $\bB_d$ with kernel $k(z,w) = \sum_{n=0}^\infty a_n \langle z,w \rangle^n$ having the following additional properties:
\begin{enumerate}
  \item \label{it:roc} the power series $\sum_{n=0}^\infty a_n t^n$
  has radius of convergence $1$, and
  \item \label{it:cnp} $\cH$ is an irreducible complete Nevanlinna-Pick space.
\end{enumerate}

The radius of
convergence of the power series appearing in condition \eqref{it:roc} is at least $1$, as $k$ is defined on $\bB_d \times \bB_d$. Demanding
that it be exactly $1$ guarantees that the functions in $\cH$ do not all extend to analytic functions on a ball of radius greater than $1$.
This condition is closely related to the notion of algebraic consistency, see \cite[Section 5]{hartz2015isom} and in particular Lemma 5.3 therein.
Finally, in relation to property \eqref{it:cnp}, the characterization of the complete Nevanlinna-Pick property takes on a particularly simple form for unitarily invariant spaces, and we record it below for convenience. It is essentially \cite[Lemma 7.33]{AM2002}. The precise statement can also be found in \cite[Lemma 2.3]{hartz2015isom}.

\begin{lemma}
  \label{L:NP-char}
  Let $\cH$ be a unitarily invariant space on $\bB_d$ with reproducing kernel
  \begin{equation*}
    k(z,w) = \sum_{n=0}^\infty a_n \langle z,w \rangle^n.
  \end{equation*}
  Then $\cH$ is an irreducible complete Nevanlinna-Pick space if and only if 
  \[
  1-\frac{1}{k(z,w)}=\sum_{n=1}^\infty b_n \langle z,w\rangle^n  \quad (z,w\in \bB_d)
  \]
  for some sequence $\{b_n\}_{n}$ of non-negative numbers.
\end{lemma}
Observe that the sequence $\{b_n\}_n$ above is in fact uniquely determined by the reproducing
kernel $k$.

\subsection{Bounded and unbounded kernels}\label{SS:an}
When working with unitarily invariant spaces, we will find it convenient to distinguish two cases, depending
on the summability of the sequence $\{a_n\}_n$, or equivalently, on the boundedness of the reproducing kernel.
\begin{itemize}[wide]

\item   
The first case corresponds to $\sum_{n=0}^\infty a_n = \infty$. Then, for every $\zeta$ in the unit sphere, we have $k(w,w) \to \infty$ as $w \to \zeta$.
  As a consequence,
  the space $\cH$ contains a function which is unbounded near $\zeta$, and the open unit ball is the natural domain of definition for the functions in $\cH$. Indeed, this can be made rigorous by saying
  that $\cH$ is an algebraically consistent space on $\bB_d$ (see \cite[Lemma 5.3]{hartz2015isom}),
  but we will not require this more precise statement.

\item The second case corresponds to $\sum_{n=0}^\infty a_n < \infty$. Then, the reproducing kernel $k$ extends
  to a continuous function $\widehat{k}$ on $\ol{\bB_d} \times \ol{\bB_d}$.
  It follows that all functions in $\cH$ extend to continuous functions
  on $\ol{\bB_d}$, and thus $\cH$ can be viewed as a reproducing kernel Hilbert space
  on $\ol{\bB_d}$ whose kernel is $\widehat{k}$.
  Indeed, if $\widehat \cH$ denotes the reproducing kernel Hilbert
  space on $\ol{\bB_d}$ with kernel $\widehat{k}$, then $\widehat \cH \big|_{\bB_d} = \cH$.
  Moreover, since $\widehat{k}$ is continuous on $\ol{\bB_d} \times \ol{\bB_d}$, the functions
  in $\widehat \cH$ are continuous on $\ol{\bB_d}$, so the coisometric
  restriction map $\widehat \cH \to \cH$ (see \cite[Section I.5]{aronszajn1950})
  is injective, and hence unitary.
  In this setting,
  the natural domain of definition for the functions in $\cH$ is the closed unit ball
  (see \cite[Lemma 5.3]{hartz2015isom} for a more precise version of this statement).
\end{itemize}

\subsection{Examples}
\label{ss:exa}
The class of unitarily invariant complete Nevanlinna-Pick spaces on the ball
contains many widely studied spaces.
For instance, for  $s \le 0$ consider the kernel
  \begin{equation*}
    \sum_{n=0}^\infty (n+1)^s \langle z,w \rangle^n
  \end{equation*}
  on the unit ball.
  Then, the associated reproducing kernel Hilbert space $\cH_s(\bB_d)$
  is a unitarily invariant complete Nevanlinna-Pick space (see \cite[Corollary 7.41]{AM2002}). If $-1 \le s \le 0$, the space $\cH_s(\bB_d)$ falls into
  the first category discussed in Subsection \ref{SS:an}, and if $s < -1$ it belongs to the second category.

  If $d =1$, then these are spaces on the unit disc $\bD$, which we simply denote by $\cH_s$.
  The scale of spaces $\{\cH_s: s \le 0\}$
  contains in particular the Hardy space ($s=0$) and the Dirichlet space
  $(s= -1)$, the latter of  which is the source of many deep problems and the target of intense current research;
 see for instance the book \cite{EKM+14} and the references therein.
 For every $d \in \bN$, the space $\cH_0(\bB_d)$ is the Drury-Arveson space $H^2_d$,
 which was mentioned in the introduction.

 A related class of examples are the Besov-Sobolev spaces $\cK_\sigma$ on $\bB_d$ with reproducing kernel
 \begin{equation*}
   \frac{1}{(1 - \langle z,w \rangle)^\sigma},
 \end{equation*}
 where $\sigma \in (0,1]$. It is well known that $\cK_\sigma$ and $\cH_{\sigma-1}$ coincide as vector spaces and that their norms are  equivalent.

\subsection{The algebras $A(\cH)$}
If $\cH$ is a unitarily invariant space on the ball such that
the coordinate functions $z_1,\ldots,z_d$ are multipliers, then we define
\begin{equation*}
  A(\cH) = \ol{\bC[z_1,\ldots,z_d]}^{|| \cdot ||} \subset \Mult(\cH).
\end{equation*}
 That is, the operator algebra $A(\cH)$ is the norm closure of the polynomials inside the multiplier algebra $\Mult(\cH)$.
 If $\cH$ is a unitarily invariant complete Nevanlinna-Pick space, then the coordinate
 functions are automatically multipliers (see the discussion
 at the beginning of \cite[Section 4]{greene2002}), so that $A(\cH)$ is always
 defined in this case.
 Algebras of the type $A(\cH)$
 will be one of our main objects of interest throughout the paper.
 
 Note that if $\cH$ is the Hardy space $H^2(\bD)$, then $A(\cH)$ is the classical disc algebra $A(\bD)$. If $\cH $ is the Drury-Arveson
space $H^2_d$, then $A(\cH)$ is the algebra $\cA_d$ discussed in the introduction, which plays a central role in multivariable operator
theory. Indeed, spurred on by Popescu's \cite{popescu1991} and Arveson's \cite{arveson1998} seminal works,  this algebra and its non-commutative counterpart have received constant and considerable attention over the years (see for instance \cite{DP1998,DRS2011,CD2016duality,CD2016abscont}). 

\subsection{Regular spaces and Toeplitz algebras}
\label{ss:regular}

As mentioned in the introduction, we will sometimes require an additional property of a unitarily invariant space on $\bB_d$ with reproducing kernel
\[
k(z,w)=\sum_{n=0}^\infty a_n \langle z,w \rangle^n,
\]
namely that
\[
\lim_{n \to \infty} a_n / a_{n+1} = 1.
\]
We will call such a space \emph{regular}.
This assumption is fairly common in the study of unitarily invariant spaces (see for example \cite[Section 4]{greene2002} or \cite[Proposition 8.5]{hartz2015isom}). All the classical spaces considered in Subsection \ref{ss:exa}
enjoy this property.
Since we usually assume that the power series $\sum_{n=0}^\infty a_n t^n$ has radius of convergence $1$,
the quantity $\lim_{n \to \infty} a_n / a_{n+1}$, if it exists, is necessarily $1$.
Thus, we regard this extra condition as a regularity assumption on the kernel, hence the terminology.
If $\cH$ is regular, then the polynomials are automatically multipliers on $\cH$ (see \cite[Corollary 4.4 (1)]{GHX04}).
In our context, the usefulness of the regularity assumption stems from the following
result, which is \cite[Proposition 4.3 and Theorem 4.6]{GHX04}. Throughout, we denote the ideal of compact operators on a Hilbert space $\cE$ by $\cK(\cE)$ and the unit sphere (the boundary of the unit ball) by $\partial \bB_d$.

\begin{theorem}[Guo-Hu-Xu]
  \label{T:essentiallynormalan}
Let $\cH$ be a regular unitarily invariant space on $\bB_d$. Then, the operator $I-\sum_{j=1}^d M_{z_j}M^*_{z_j}$ is compact, and there is a short exact sequence
\begin{equation*}
  0 \longrightarrow \cK(\cH) \longrightarrow C^*(A(\cH)) \longrightarrow C(\partial \bB_d) \longrightarrow 0,
\end{equation*}
where the first map is the inclusion and the second map sends $M_{z_j}$ to the coordinate
function $z_j$ for each $1 \le j \le d$.
\end{theorem}

\section{Preliminaries on dilations}
\label{sec:prelim_dilations}

In this section, we recall a few basic notions from dilation theory. For background material,
the reader is referred to the book \cite{paulsen2002}.

\subsection{Dilations of operator tuples}
Let $T = (T_1,\ldots,T_d)$ be a tuple of commuting operators on a Hilbert space $\cH$,
let $\cL$ be a Hilbert space which contains $\cH$ as a closed subspace and let $S = (S_1,\ldots,S_d)$ be a tuple
of commuting operators on $\cL$. We say that \emph{$S$ is a dilation of $T$}, or that \emph{$T$ dilates to $S$},
if
\begin{equation*}
  p(T) = P_{\cH} p(S) \big|_{\cH} \quad \text{ for all } p \in \bC[z_1,\ldots,z_d].
\end{equation*}
By a theorem of Sarason (see, for example \cite[Theorem 7.6]{paulsen2002}), this happens
if and only if $\cL$ admits a decomposition $\cL = \cL_{-} \oplus \cH \oplus \cL_{+}$ with respect to which we can write\begin{equation*}
  S_k =
  \begin{bmatrix}
    \ast & 0 & 0 \\
    \ast & T_k & 0 \\
    \ast & \ast & \ast
  \end{bmatrix} \quad (k=1,\ldots,d).
\end{equation*}

We say that \emph{$S$ is an extension of $T$}, or that \emph{$T$ extends to $S$},
if the space $\cL_+$ above is absent. Likewise, we say that \emph{$S$ is a coextension of $T$},
or that \emph{$T$ coextends to $S$}, if the space $\cL_-$ is absent, that is, if
\begin{equation*}
  S_k =
  \begin{bmatrix}
    T_k & 0 \\
    \ast & \ast
  \end{bmatrix} \quad (k=1,\ldots,d).
\end{equation*}
Clearly $S^*$ is an extension of $T^*$ if and only if $S$ is a coextension of $T$.

\subsection{Dilations of representations}
While dilations of operators are our main motivation, it will be beneficial to adopt a more
operator algebraic point of view. Let $\cA$ be a unital operator algebra. 
A \emph{representation of $\cA$} is a unital homomorphism
\begin{equation*}
  \rho: \cA \to B(\cH)
\end{equation*}
for some Hilbert space $\cH$.
Let $\cL \supset \cH$ be a larger Hilbert space, let $\cB \supset \cA$ be a larger operator
algebra and let
$\sigma: \cB \to B(\cL)$ be another representation. We say that \emph{$\rho$ dilates to $\sigma$}
if
\begin{equation*}
 \rho(a) =  P_{\cH} \sigma(a) \big|_{\cH} \quad \text{ for all } a \in \cA.
\end{equation*}
By the aforementioned result of Sarason, this happens if and only if $\cL$ admits a
decomposition $\cL = \cL_{-} \oplus \cH \oplus \cL_+$ with respect to which we can write
\begin{equation*}
    \sigma(a) =
    \begin{bmatrix}
      \ast & 0 & 0 \\
      \ast & \rho(a) & 0\\
      \ast & \ast & \ast
    \end{bmatrix} \quad \text{ for all } a \in \cA.
\end{equation*}
We say that \emph{$\rho$ extends to $\sigma$} if $\cL_+$ is absent,
and \emph{$\rho$ coextends to $\sigma$} if $\cL_-$ is absent.

Of particular interest is the case when $\cB$ is a unital $C^*$-algebra which contains $\cA$.
In this case, we say that \emph{$\rho$ admits a $\cB$-dilation} if $\rho$ dilates to a $*$-representation
of $\cB$. Similarly, we say that \emph{$\rho$ admits a $\cB$-coextension}
if $\rho$ coextends to a $*$-representation of $\cB$.

Let $n\in \bN$ and denote by $M_n(\cA)$ the algebra of $n\times n$ matrices with entries from $\cA$.
A linear map $\psi: \cA \to B(\cH)$ induces another linear map
$\psi^{(n)}: M_n(\cA) \to M_n(B(\cH))$, defined by applying $\psi$ entry-wise. Then, $\psi$ is said to be \emph{completely contractive} (respectively, \emph{completely isometric})
if the induced map $\psi^{(n)}$
is contractive (respectively, isometric) for every $n \in \bN$. 
Using these ideas, Arveson \cite{arveson1969} characterized abstractly when a representation of $\cA$
admits a $C^*(\cA)$-dilation (see also \cite[Corollary 7.7]{paulsen2002}).

\begin{theorem}[Arveson]\label{T:arv}
  Let $\cB$ be a unital $C^*$-algebra, let $\cA \subset \cB$
  be a unital subalgebra of $\cB$ and let $\rho: \cA \to B(\cH)$ be a representation
  of $\cA$. Then $\rho$ admits a $\cB$-dilation if and only if it is completely contractive.
\end{theorem}

From an operator theoretic point of view, it may seem more natural to
rephrase results about coextensions in terms of extensions
of the adjoints. In our case, however,
the operator algebraic formulation makes it more convenient to work with the coextensions directly.

\subsection{Coextensions of unital completely positive maps}

An \emph{operator system} is a unital
self-adjoint subspace of a unital $C^*$-algebra. A linear map $\psi$ between operator systems
is said to be \emph{positive} if it maps positive elements to positive elements, and \emph{completely
positive} if the induced map $\psi^{(n)}$ is positive for every $n \in \bN$. 

We will occasionally require a slightly more detailed description of the dilation in Theorem \ref{T:arv}. If $\cA$ is a unital subalgebra of a unital $C^*$-algebra $\cB$ and $\rho: \cA \to B(\cH)$
is a unital completely contractive representation, then $\rho$ extends uniquely
to a unital completely positive linear map $\widetilde \rho: \cA + \cA^* \to B(\cH)$ \cite[Proposition 3.5]{paulsen2002}.
Arveson's extension theorem \cite[Theorem 7.5]{paulsen2002} implies that $\widetilde \rho$ extends
to a unital completely positive linear map  on $\cB$,
and Stinespring's dilation theorem \cite[Theorem 4.1]{paulsen2002} then yields the desired $\cB$-dilation of $\rho$.

To obtain coextensions as opposed to mere dilations, we will use a device of Agler, which has
its roots in his work on hereditary polynomials \cite{agler1982}. In our context, the relevant
statement is the following lemma (see \cite[Theorem 1.5]{agler1982} and its proof).

\begin{lemma}[Agler]
  \label{L:hered_coinvariant}
  Let $\cB$ be a unital $C^*$-algebra, let
  $\psi: \cB \to B(\cH)$ be a unital completely positive linear map and let $\pi:
  \cB \to B(\cL)$ be a $*$-representation such that
  \[
  \psi(T)=P_{\cH}\pi(T)|_{\cH}
  \]
  for every $T\in \cB$.  For each $T \in \cB$,
  the following assertions are equivalent:
  \begin{enumerate}[label=\normalfont{(\roman*)}]
    \item $\psi(T T^*) = \psi(T) \psi(T)^*$,
    \item $\cH$ is coinvariant for $\pi(T)$.
  \end{enumerate}
\end{lemma}
It should be noted that given a unital completely positive map $\psi$ as in the lemma above, there always exists a $*$-representation $\pi$ with 
  \[
  \psi(T)=P_{\cH}\pi(T)|_{\cH} \quad (T\in \cB).
  \]
 Indeed, this is the content of Stinespring's dilation theorem.

\subsection{Representations of $C^*$-algebras containing the compact operators}

We will frequently make use of the
following standard fact about representations of $C^*$-algebras containing
the compact operators.
\begin{lemma}
  \label{lem:rep_split}
  Let $\mathcal B\subset B(\cH)$ be a unital $C^*$-algebra which contains the ideal $\cK$ of compact operators on $\cH$, and let
  $\pi$ be a unital $*$-representation of $\cB$. Then $\pi$ splits
  as a direct sum $\pi = \pi_1 \oplus \pi_2$, where $\pi_1$ is unitarily equivalent to
  a multiple of the identity representation and $\pi_2$ annihilates $\cK$.
\end{lemma}

\begin{proof}
  This follows from the discussion preceding Theorem I.3.4 in \cite{Arveson76}
  and the fact that every non-degenerate representation of $\cK$ is unitarily equivalent to a multiple of the identity
  representation.
\end{proof}

\subsection{Boundary representations and the $C^*$-envelope}
\label{S:prelimuep}
Let $\cA \subset B(\cH)$ be a unital operator algebra and $\psi:\cA\to B(\cE)$ be a unital completely contractive linear map.
As explained above, there exists a unital completely positive linear map $\pi:C^*(\cA)\to B(\cE)$ which agrees with $\psi$ on $\cA$. We say that $\psi$ has the \emph{unique extension property} if there is a unique such extension,
and this extension is a $*$-homomorphism.
An irreducible $*$-representation $\pi$ of $C^*(\cA)$ is said to be \emph{boundary representation for $\cA$} if $\pi|_{\cA}$ has the unique extension property. This idea was introduced by Arveson \cite{arveson1969} in analogy with the classical Choquet boundary of a uniform algebra.
It follows from \cite[Theorem 2.1.2]{arveson1969} that this notion does not depend on the concrete
embedding of $\cA$ into $B(\cH)$, and is thus a completely isometric invariant of $\cA$.
The algebra $\cA$ is said to be \emph{hyperrigid} \cite{Arveson11} if the restriction $\pi \big|_{\cA}$ has the unique extension property for every unital $*$-representation $\pi$ of $C^*(\cA)$.

Any operator algebra $\cA$ has sufficiently many
boundary representations in the sense that there exists a set of boundary representations whose
direct sum is completely isometric on $\cA$.
This was first proved by Arveson \cite{arveson2008} in the separable case,
and later established in full generality by Davidson and Kennedy \cite{DK2015}.

In general, there are many $C^*$-algebras which are generated by a completely isometric copy
of a unital operator algebra $\cA$, and accordingly it is natural to look for the smallest one in the following precise sense. The \emph{$C^*$-envelope of $\cA$} is a $C^*$-algebra
$C^*_e(\cA)$ together with a unital completely isometric linear map $\iota: \cA \to C^*_e(\cA)$ with the following two properties:
\begin{enumerate}
\item $C^*_e(\cA) = C^*(\iota(\cA))$,

\item whenever $j:\cA \to \cB$ is another unital completely isometric linear map and  $\cB=C^*(j(\cA))$, then
there exists a surjective $*$-homomorphism $\pi: \cB \to C^*_e(\cA)$ such that $\pi \circ j = \iota$.
\end{enumerate}
It follows from \cite[Theorem 2.1.2]{arveson1969}  
that if $\iota$ is a direct sum of boundary representations for $\cA$ which is completely
isometric on $\cA$, then
$C^*(\iota(\cA))$ is the $C^*$-envelope of $\cA$ (see also \cite{ArvesonUEP}).
In particular, if the identity representation of $C^*(\cA)$ on $B(\cH)$ is a boundary
representation for $\cA$, then $C^*_e(\cA) = C^*(\cA)$.
Even before it was known that there are sufficiently many boundary representations, these
ideas were used by Dritschel and McCullough to construct the $C^*$-envelope \cite{dritschel2005},
but the existence of such an object was proved yet earlier using different methods in \cite{hamana1979}.

\subsection{A sufficient condition for the unique extension property}
While we know that boundary representations are always plentiful as discussed above, determining when a given unital completely contractive representation has the unique extension property remains a difficult task in general. There are some instances however where this is possible, and we proceed to describe one that will be useful for us. For that purpose, recall that a commuting family $T = (T_1,T_2,\ldots)$
of operators on a Hilbert space is said to be a \emph{row contraction} if
\begin{equation*}
  \sum_{n} T_n T_n^* \le I,
\end{equation*}
where the series converges in the strong operator topology. Moreover,
a commuting tuple $U=(U_1,U_2,\ldots)$ of operators on a Hilbert space is said to be a
\emph{spherical unitary} if each $U_n$ is normal and
\begin{equation*}
  \sum_{n} U_n U_n^* = I,
\end{equation*}
where the series converges in the strong operator topology as well.
The following result follows from a theorem of Richter and Sundberg \cite{RS10} and is a generalization of \cite[Corollary 2.2]{KS2015}; the reader may also want to compare with \cite[Theorem 3.1.2]{arveson1969}. 

\begin{theorem}
  \label{T:sphunitary_uep}
  Let $T= (T_1,T_2,\ldots)$ be a commuting row contraction and let $\cA_T$ be
  the unital norm closed operator algebra generated by $\{T_n: n \ge 1\}$. Let $\rho: \cA_T \to B(\cH)$
  be a unital completely contractive linear map such that $(\rho(T_1),\rho(T_2),\ldots)$
  is a spherical unitary. Then $\rho$ has the unique extension property.
\end{theorem}

\begin{proof}
  Let $\widehat \rho:C^*(\cA_T)\to B(\cH)$ be a unital completely positive extension of $\rho$, and let
  $\pi: C^*(\cA_T) \to B(\cL)$ be a Stinespring dilation of $\widehat \rho$.
  We will show that $\cH$ is reducing for $\pi(C^*(\cA_T))$, which implies
  that $\widehat \rho$ is itself a $*$-homomorphism and thus the unique
  completely positive extension of $\rho$.

First, we may use the Schwarz inequality for completely positive maps to obtain that $\rho(T_n)\rho(T_n)^*\leq \widehat{\rho}(T_nT_n^*)$ for every $n$. Thus we have
  \begin{equation*}
    I = \sum_n  \rho(T_n)\rho(T_n)^* \le \sum_n \widehat \rho(T_n T_n^*)
    \le I,
  \end{equation*}
  where the last inequality is due to the fact that $T$ is a row contraction while
  $\widehat \rho$ is unital and completely positive.
  Thus, equality holds throughout, so $\widehat \rho(T_n) \widehat \rho(T_n)^* = \widehat \rho(T_n T_n^*)$
  for all $n$. Lemma \ref{L:hered_coinvariant} now implies
  that $\cH$ is coinvariant for each $\pi(T_n)$,
  hence 
  \[
  \pi(T_n)^*|_\cH=\rho(T_n)^*
  \]
  for each $n$. We see that the commuting tuple $(\pi(T_1),\pi(T_2),\ldots)$ is a row contraction, and $(\pi(T_1)^*,\pi(T_2)^*,\ldots)$
  is an extension of the spherical unitary $(\rho(T_1)^*,\rho(T_2)^*,\ldots)$.
  It follows from \cite[Theorem 2.2]{RS10} that $\cH$ is also
  invariant for $\pi(T_n)$ for each $n $ (although that theorem is only
  stated for finite tuples of operators, it is clear that the proof
  extends verbatim to the case of infinite tuples). Therefore,
  $\cH$ reduces each $\pi(T_n)$, and hence reduces $\pi(C^*(\cA_T))$.
\end{proof}

We remark that the proof establishes in fact the following more general statement. Let $T = (T_1,T_2,\ldots)$
be a commuting row contraction and let $\cS_T$ denote the operator system generated by $T$. If $\rho: \cS_T \to
B(\cH)$ is a
unital completely positive map such that $(\rho(T_1),\rho(T_2),\ldots)$ is a
spherical unitary, then $\rho$ has the unique extension property.

\section{Completely contractive inclusions of multiplier algebras}\label{S:incl}

The aim of the next two sections is to address Question \ref{Q:intro1} for algebras of multipliers
of reproducing kernel Hilbert spaces.
Let $\cH_1$ and $\cH_2$ be two reproducing kernel Hilbert spaces
on the same set $X$ with respective kernels $k_1$ and $k_2$. In this section,
we aim to determine when $\Mult(\cH_1)$ is contained  completely contractively  in $\Mult(\cH_2)$. When this occurs,
the inclusion may be regarded as a particular representation of $\Mult(\cH_1)$.

There is a basic sufficient condition for the inclusion to be completely contractive, namely that $k_2/k_1$ be positive semi-definite on $X$. Indeed, this is an easy consequence of the Schur product theorem (see, for instance, the implication (ii) $\Rightarrow$ (i)
in \cite[Corollary 3.5]{hartz2015isom}). Interestingly, this positivity condition on the kernels is also sufficient for the inclusion  to admit a special kind of $B(\cH_1)$-coextension. This non-trivial fact is known, see \cite[Theorem 2]{douglas2012dilation} and
\cite[Corollary 6.2]{KSS+16}. Since the setting of these papers is slightly different, we provide a sketch of the proof of a statement adapted to our purposes.

\begin{theorem}\label{T:douglas}
Let $\cH_1$ and $\cH_2$ be reproducing kernel Hilbert spaces on the same set $X$ with respective kernels $k_1, k_2$.
Assume that $k_1$ is non-vanishing and that $k_2/k_1$ is positive semi-definite.
Then $\Mult(\cH_1) \subset \Mult(\cH_2)$ and
there exists a Hilbert space $\cE$ along with an isometric map $V:\cH_2\to \cH_1\otimes \cE$ with the property that
\[
V^*(M_{\phi}^{\cH_1}\otimes I_{\cE})=M_{\phi}^{\cH_2}V^*,
\]
for every $\phi \in \Mult(\cH_1)$. 
\end{theorem}

\begin{proof}
  Since $k_2/k_1$ is positive, a well-known factorization theorem for positive semi-definite kernels (see \cite[Theorem 2.53]{AM2002})
  implies that there exists a Hilbert space $\cE$ and a function $f: X \to \cE$ such that
  \begin{equation*}
    k_2(x,y) = k_1(x,y) \langle f(y),f(x) \rangle
  \end{equation*}
  for $x,y \in X$. Therefore, there exists
  an isometry $V: \cH_2 \to \cH_1 \otimes \cE$ which satisfies
  \begin{equation*}
    V k_2(\cdot,x) = k_1(\cdot,x) \otimes f(x)
  \end{equation*}
  for all $x \in X$. A straightforward computation then shows that
for $\varphi \in \Mult(\cH_1)$ we have
  \begin{equation*}
    (M_\varphi^{\cH_1} \otimes 1_{\cE})^* V k_2(\cdot,x) = V \ol{\varphi(x)} k_2(\cdot,x)
  \end{equation*}
  for all $x \in X$. In particular, the range of $V$ is coinvariant for $M^{\cH_1}_\varphi \otimes I_{\cE}$, and
  the bounded operator $V^* (M_\varphi^{\cH_1} \otimes 1_{\cE})^* V$ on $\cH_2$ has $k_2(\cdot,x)$
  as an eigenvector with eigenvalue $\ol{\varphi(x)}$. Thus, we conclude that
  $\varphi \in \Mult(\cH_2)$ and that
  $V^* (M_\varphi^{\cH_1} \otimes 1_{\cE}) = M_\varphi^{\cH_2} V^*$.
\end{proof}

In the presence of the complete Nevanlinna-Pick property,
the last result fits in a chain of equivalences.
In particular, we obtain a positive answer to Question \ref{Q:intro1} for a special class
of representations of multiplier algebras of complete Nevanlinna-Pick spaces.

\begin{theorem}\label{T:equiv}
Let $\cH_1$ and $\cH_2$ be two reproducing kernel Hilbert spaces on the same set $X$
with reproducing kernels $k_1,k_2$, respectively. Suppose that $\cH_1$
is an irreducible complete Nevanlinna-Pick space. Then the following assertions are equivalent.
\begin{enumerate}[label=\normalfont{(\roman*)}]
  \item The inclusion $\Mult(\cH_1) \hookrightarrow \Mult(\cH_2)$ is completely contractive.
  \item The inclusion $\Mult(\cH_1)\hookrightarrow \Mult(\cH_2)$ admits a $B(\cH_1)$-coextension.
  \item $\Mult(\cH_1) \subset \Mult(\cH_2)$ and there exists a
    Hilbert space $\cE$ and an isometry $V: \cH_2 \to \cH_1 \otimes \cE$ such that
    $V^*(M_{\phi}^{\cH_1}\otimes I_{\cE})=M_{\phi}^{\cH_2}V^*$
    for every $\varphi \in \Mult(\cH_1)$.
  \item $k_2/k_1$ is positive semi-definite on $X$.
\end{enumerate}
\end{theorem}

\begin{proof}
   Theorem \ref{T:douglas} shows that (iv) implies (iii).
   
   Suppose that (iii) holds and let
  \begin{equation*}
    \pi: B(\cH_1) \to B(\cH_1 \otimes \cE), \quad T \mapsto T \otimes I_{\cE}.
  \end{equation*}
  Then $\pi$ is a $*$-representation, and if we identify the space $\cH_2$ with
 $V\cH_2\subset \cH_1 \otimes \cE$, it becomes coinvariant
  under $\pi(M_\varphi^{\cH_1})$ and
  \[
    (M_{\phi}^{\cH_2})^*= (M_\varphi^{\cH_1} \otimes I_{\cE})^* |_{\cH_2} =
  \pi(M_{\phi}^{\cH_1})^*|_{\cH_2}
  \]
  for $\varphi \in \Mult(\cH_1)$. Thus, (ii) holds.

(ii) trivially implies (i). 

Finally,  we show that (i) implies (iv). Observe that both the hypothesis and the conclusion are invariant
  under rescaling of the kernel $k_1$, so we may assume that
  $k_1$ is normalized at a point (see \cite[Section 2.6]{AM2002}). In this setting, the statement  follows from the implication (i) $\Rightarrow$ (ii) of \cite[Corollary 3.5]{hartz2015isom}
  and its proof.
\end{proof}

For a given algebra of functions $\cM$ on a set $X$, there may be many different reproducing
kernel Hilbert spaces $\cH$ on $X$ whose multiplier algebra is $\cM$. For instance, the algebra
$H^\infty(\bD)$ of bounded analytic functions on $\bD$ is the multiplier algebra of the Hardy space, of the Bergman space and indeed of
many more spaces on the unit disc. If one of these spaces is an irreducible complete
Nevanlinna-Pick space $\cH_1$, then the
implication (i) $\Rightarrow$ (iii) of Theorem \ref{T:equiv} shows that
among all those spaces, $\cH_1$ is a particularly good choice
from the point of view of dilation theory. Indeed, for every $\varphi \in \cM$ and each of the spaces $\cH$,
the operator $M_\varphi^{\cH}$ coextends to $M_\varphi^{\cH_1} \otimes I_{\cE}$.
Therefore, given $\varphi \in \cM$, the collection of operators $T$
which can be modelled by $M_\varphi^{\cH}$, in the sense that $T$ coextends
to $M_\varphi^{\cH} \otimes I_{\cE}$, is maximal if $\cH = \cH_1$.

We also observe that in the proof of Theorem \ref{T:equiv}, the Nevanlinna-Pick property of $\cH_1$
 was only used to show the implication (i) $\Rightarrow$ (iv), while the implications
(iv) $\Rightarrow$ (iii), (iii)  $\Rightarrow$ (ii) and (ii)  $\Rightarrow$ (i) hold without it.
Moreover,
in many cases of interest, the equivalence of (iii) and (iv) also holds without
the Nevanlinna-Pick assumption on $\cH_1$,
see \cite[Theorem 2]{douglas2012dilation} and \cite[Corollary 6.2]{KSS+16}.
On the other hand, the following example shows that without the Nevanlinna-Pick assumption on $\cH_1$, the implications (i) $\Rightarrow$ (ii) (and therefore also (i) $\Rightarrow$ (iii) and (i) $\Rightarrow$ (iv))
may fail.

\begin{example}
  \label{E:Hardy_Bergman}
  Let $\cH_1 = A^2$ be the Bergman space on the unit disc whose reproducing kernel is
  \[
  k_1(z,w)= \frac{1}{(1-z \ol{w})^2}.
  \]
  Let $\cH_2 = H^2$ be the classical Hardy space on the unit disc.
  Then, both $\Mult(A^2)$ and $\Mult(H^2)$ coincide with the algebra $H^\infty(\bD)$. In particular,
the inclusion
  \begin{equation*}
  \Mult(A^2)\hookrightarrow \Mult(H^2) 
  \end{equation*}
  is completely contractive. Let $\pi: B(A^2) \to B(\cL)$ be a $*$-homomorphism which dilates the inclusion map
  (there always exists such a dilation by Theorem \ref{T:arv}).

  We claim that $H^2$ cannot be coinvariant for $\pi( \Mult(A^2))$, so that the inclusion does not admit a $C^*(\Mult(A^2))$-coextension. To this end,
  observe that the Bergman shift $B = M_z^{A^2}$ satisfies
  \begin{equation*}
    B^2 (B^*)^2 - 2 B B^* + I \ge 0.
  \end{equation*}
This relation is preserved by unital $*$-homomorphisms as well as by compressions to coinvariant subspaces. Let $S=M_z^{H^2}$ denote the Hardy shift. Since $S=P_{H^2}\pi(B)|_{H^2}$, if $H^2$ were coinvariant for $\pi(\Mult(A^2))$ we would have 
  \begin{equation*}
    S^2 (S^*)^2 - 2 S S^* + I  \ge 0,
  \end{equation*}
  which is easily verified to be false.
  Incidentally, this argument also shows that $S$ does not coextend to $B \otimes I_{\cE}$
  for some Hilbert space $\cE$.
 Finally, we note that
  \begin{equation*}
   k_2(z,w)/k_1(z,w) = 1 - z \ol{w},
 \end{equation*}
 which is not positive semi-definite on $\bD$, in accordance with Theorem \ref{T:douglas}.

 On the other hand, if we interchange the roles of Bergman space and Hardy space and set
 $\cH_1 = H^2$ and $\cH_2 = A^2$, then $\cH_1$ is a complete Nevanlinna-Pick space, the quotient  
 \begin{equation*}
   k_2 (z,w) / k_1(z,w) = \frac{1}{1 - z \ol{w}} = k_1(z,w)
 \end{equation*}
  is positive semi-definite, and thus
 the inclusion $ \Mult(H^2) \to \Mult(A^2)$ admits a $C^*(\Mult(H^2))$-coextension by Theorem \ref{T:equiv}.
 \qed
\end{example}

This example shows that the complete Nevanlinna-Pick property cannot simply be removed from the statement of Theorem \ref{T:equiv}. In fact, the next result shows that the validity of the implication
(i) $\Rightarrow$ (iv) characterizes complete Nevanlinna-Pick spaces.

\begin{theorem}\label{T:NPchar}
  Let $\cH_1$ be an irreducible reproducing kernel Hilbert space on a set $X$ with kernel $k_1$.
  Then the following assertions are equivalent.
  \begin{enumerate}[label=\normalfont{(\roman*)}]
    \item $\cH_1$ is a complete Nevanlinna-Pick space.
    \item Whenever $\cH_2$ is another reproducing kernel Hilbert space on $X$ with kernel $k_2$ such that the inclusion $\Mult(\cH_1)\hookrightarrow \Mult(\cH_2)$    is completely contractive, then the function $k_2 / k_1$ is positive semi-definite on $X$.
  \end{enumerate}
\end{theorem}

\begin{proof}
 In view of Theorem \ref{T:equiv}, it remains to show that (ii) implies (i). Suppose thus that (ii) holds.
  By rescaling $k_1$ if necessary,
  we may without loss of generality assume that $k_1$ is normalized at a point
  $x_0 \in X$ so that $k_1(\cdot, x_0)=1$ (see \cite[Section 2.6]{AM2002} and note that the
  rescaling procedure preserves the complete Nevanlinna-Pick property).
  Define
    \begin{equation*}
    \cH_0 = \{f \in \cH_1 : f(x_0) = 0 \}.
  \end{equation*}
  Then $\cH_0$ is a closed subspace of $\cH_1$ and, with the Hilbert space structure inherited from
  $\cH_1$, it is a reproducing
  kernel Hilbert space on $X$.
  Since $1 = k_1(\cdot,x_0)$, we see that the reproducing
  kernel of $\cH_0$ is $k_0 = k_1 - 1$.
Some care must be taken here however, since the functions in $\cH_0$ share the common zero $x_0$,
  so that a multiplier on $\cH_0$ is not uniquely determined by its associated multiplication operator.
  Thus, $\cH_0$ does not fit into the framework of Section \ref{S:prelimRKHS}.
  
  To remedy this situation, we define for $0 < \varepsilon \le 1$ the kernel
  \begin{equation*}
    k_\varepsilon = k_1 - (1 - \varepsilon) = k_0 + \varepsilon
  \end{equation*}
  on $X$ and we let $\cH_\varepsilon$ be the reproducing kernel Hilbert space
  with kernel $k_\varepsilon$. Clearly, the kernel $k_\varepsilon$ does not vanish on the diagonal of $X\times X$, so
  that the functions in $\cH_\varepsilon$ have no common zeros. We claim that the inclusion $\Mult(\cH_1) \hookrightarrow \Mult(\cH_\varepsilon)$ is completely contractive. Let $r\in \bN$ and let $\Phi\in M_r(\Mult(\cH_1))$ with norm at most $1$. Then, $\Phi$ can be regarded as a contractive
  $M_r(\bC)$-valued multiplier on $\cH_1$, so that
  \begin{equation*}
    L_1(z,w) = k_1(z,w) (I_{\bC^r} -  \Phi(z) \Phi(w)^*)
  \end{equation*}
  is a positive semi-definite $M_r(\bC)$-valued kernel. It is clear that the restriction of $\Phi$ to the invariant subspace $\cH_0^{(r)}$ also has norm at most $1$, hence
  \begin{equation*}
    L_0(z,w) = k_0(z,w) (I_{\bC^r} - \Phi(z) \Phi(w)^*)
  \end{equation*}
  is also positive semi-definite. Therefore,
  \begin{equation*}
    k_\varepsilon(z,w) ( I_{\bC^r} - \Phi(z) \Phi(w)^*)
    = \varepsilon L_1(z,w) + (1 - \varepsilon) L_0(z,w)
  \end{equation*}
  is positive semi-definite, so that $\Phi$ belongs to the unit ball of $M_r(\Mult(\cH_\varepsilon ))$,
  which proves the claim.

Using (ii) with $\cH_2 = \cH_\varepsilon$ implies that
  \begin{equation*}
    1 - (1 - \varepsilon) \frac{1}{k_1} = \frac{k_\varepsilon}{k_1}
  \end{equation*}
  is positive semi-definite on $X$. Taking the limit as $\varepsilon \to 0$, we
  conclude that $1 - 1/k_1$ is positive semi-definite on $X$, so that $k_1$ is a complete Nevanlinna-Pick kernel
  by \cite[Theorem 7.31]{AM2002}.
\end{proof}

We close this section by mentioning a different interpretation of the condition (ii) in Theorem \ref{T:equiv},
which comes from Agler's work on hereditary polynomials.
Lemma \ref{L:hered_coinvariant} shows that we may reformulate the
condition as follows:
there exists a unital completely positive linear map $\Psi: B(\cH_1) \to B(\cH_2)$ such
 that
    \begin{equation*}
      \Psi(M_\varphi^{\cH_1} (M_\phi^{\cH_1})^*) = M_\varphi^{\cH_2} (M_\phi^{\cH_2})^*
          \end{equation*}
  for all $\varphi\in \Mult(\cH_1)$. In addition to Agler, this condition was also
  considered by Arveson \cite[Definition 6.1]{arveson1998}. In fact, it is possible to prove the implication (iv) $\Rightarrow$ (ii) in Theorem \ref{T:equiv}
 in a way similar to \cite{agler1982}, but this approach
does not yield the information that the $B(\cH_1)$-coextension can be chosen to be unitarily equivalent to a multiple
of the identity representation.

\section{Coextensions for unitarily invariant spaces}\label{S:coextue}
Whereas Section \ref{S:incl} was devoted to completely contractive inclusions of multiplier algebras of reproducing kernel Hilbert spaces, for the rest of the paper we will be interested in general representations of algebras of multipliers. In this section we once more examine Question \ref{Q:intro1}, this time for the algebras $A(\cH)$ associated
to unitarily invariant complete Nevanlinna-Pick spaces on the ball. This class of spaces is more restricted  than that considered in Section \ref{S:incl}, but the results we obtain about it are much stronger.

We will make use of a device which has its roots in the work of Agler \cite{agler1982}.
He showed that if $\cH$ is a reproducing kernel Hilbert space of analytic functions
on $\bD$ with kernel $k$, then under
suitable assumptions,
every operator $T$ which satisfies $\sigma(T) \subset \bD$ and
$1/k(T,T^*) \ge 0$ coextends to a
direct sum of copies of $M_z^{\cH}$ \cite[Theorem 2.3]{agler1982}.
The interpretation of $1/k(T,T^*)$ requires some care in the general case,
but we note that if $k$ is the kernel of the Hardy space $H^2(\bD)$
then $1/k(T,T^*) = I- T T^*$, so the positivity assumption simply means that $T$ is a contraction. These
ideas were extended to spaces on higher dimensional domains $\Omega \subset \bC^d$
by Agler and McCarthy \cite{AM00a}, by Ambrozie, Engli\v{s} and M\"uller \cite{ambrozie2002}
and by Arazy and Engli\v{s} \cite{AE2003}. As alluded to above, one difficulty with this
approach is to make sense of the expression $1/k(T,T^*)$ for general kernels $k$ and tuples
of commuting operators $T$.
In \cite{ambrozie2002}, a good definition is available whenever the Taylor spectrum of $T$ is contained in $\Omega$, or the function $1/k$ is a polynomial. This setting is generalized in \cite{AE2003} using
the notion of a $1/k$-calculus. Fortunately,
for unitarily invariant complete Nevanlinna-Pick spaces, there is
a straightforward way to interpret the condition $1/k(T,T^*) \ge 0$.

Let $\cH$ be a unitarily invariant complete Nevanlinna-Pick space on $\bB_d$ with kernel $k$ (see Subsection \ref{SS:uispaces} for details). Then, for $z,w\in \bB_d$ we have that
\[
k(z,w)=\sum_{n=0}^\infty a_n\langle z,w \rangle^n, \quad 1-\frac{1}{k(z,w)}=\sum_{n=1}^\infty b_n \langle z,w\rangle^n
\]
for some sequences $\{a_n\}_n, \{b_n\}_n$ of non-negative numbers with $a_0=1$ and $ a_n>0$ for every $n\in \bN$
(see Lemma \ref{L:NP-char}).
Alternatively, note that for $z,w\in \bB_d$ we have
\[
1-\frac{1}{k(z,w)}=\sum_{n=1}^\infty b_n  \sum_{|\alpha|=n} \binom{n}{\alpha} z^\alpha \ol{w}^\alpha
\]
where for a multi-index of non-negative integers $\alpha=(\alpha_1,\ldots, \alpha_d)$ we put
\[
\binom{n}{\alpha}=\frac{n!}{\alpha_1!\cdots \alpha_d! }.
\]
Given a commuting tuple of operators $T = (T_1,\ldots,T_d)$ on a Hilbert space,
we write $1/k(T,T^*)\geq 0$ to mean that
\begin{equation*}
  \sum_{n=1}^N b_n \sum_{|\alpha|=n} \binom{n}{\alpha} T^\alpha (T^*)^\alpha \le I
\end{equation*}
for all $N \in \bN$. Equivalently, $1/k(T,T^*)\geq 0$ if and only if
\begin{equation*}
  I - \sum_{n=1}^\infty b_n \sum_{|\alpha|=n} \binom{n}{\alpha} T^\alpha (T^*)^\alpha \ge 0
\end{equation*}
and the series converges in the strong operator topology. We remark that a similar idea to define $1/k(T,T^*) \ge 0$
already appeared in \cite[Section 3]{AE2003}.

\begin{remark}
  \label{rem:AEM}
The meaning that we assign to the inequality $1/k(T,T^*) \ge 0$ is consistent with that of Ambrozie-Engli\v{s}-M\"uller. This is manifest when $1/k$ is a polynomial. The other case that these authors consider is when $\sigma(T) \subset \bB_d$,
  where $\sigma(T)$ denotes the Taylor spectrum of $T$ (the book
  \cite{muller2007} is a standard reference on the topic of the Taylor spectrum). They define an operator $S = 1/k(T,T^*)$ by means of the Taylor functional calculus
  (see the discussion preceding \cite[Lemma 3]{ambrozie2002}). We claim that $S\geq 0$  if and only if $1/k(T,T^*) \ge 0$ in our sense.

First  observe that the series $\sum_{n=1}^\infty b_n \langle z,w \rangle^n$ converges
  uniformly on $\ol{\bB_d} \times \ol{\bB_d}$, as $\sum_{n=1}^\infty b_n \le 1$.
  Continuity of the Taylor functional calculus (see the discussion after  \cite[Lemma 3]{ambrozie2002})
  therefore shows that
  \begin{equation*}
    S = I - \sum_{n=1}^\infty b_n \sum_{|\alpha|=n} \binom{n}{\alpha} T^\alpha (T^*)^\alpha,
  \end{equation*}
  where the series converges in norm. In particular, $S$ is a positive operator if and only if
  \begin{equation*}
    \sum_{n=1}^\infty b_n \sum_{|\alpha| = n} \binom{n}{\alpha} T^\alpha (T^*)^\alpha \le I,
  \end{equation*}
  which is our definition of $1/k(T,T^*) \ge 0$.
  \qed
\end{remark}

We now record a simple but important fact (cf. the proof of \cite[Proposition 3.1]{AE2003}).

\begin{lemma}\label{L:1/k}
Let $\cH$ be a unitarily invariant complete Nevanlinna-Pick space on $\bB_d$ with kernel $k$ and let $M_z =  (M_{z_1},\ldots,M_{z_d})$. Then, $1/k(M_z,M_z^*)\geq 0$.
\end{lemma}
\begin{proof}
  Testing on linear combinations of kernel functions, it is routine to check that for each $N \in \bN$,
  the operator inequality
\[
I-\sum_{n=1}^N b_n \sum_{|\alpha|=n} \binom{n}{\alpha} M_z^\alpha (M_z^*)^\alpha\geq 0
\]
is equivalent to the function
\[
k(z,w) \left( 1 - \sum_{n=1}^N b_n \sum_{|\alpha|=n} \binom{n}{\alpha} z^\alpha \ol{w}^\alpha \right)
\]
being positive semi-definite on $\bB_d\times \bB_d$. But since $b_n \ge 0$ for all $n \in \bN$,
this function dominates
\[
k(z,w) \left( 1 - \sum_{n=1}^\infty b_n \sum_{|\alpha|=n} \binom{n}{\alpha} z^\alpha \ol{w}^\alpha \right)
\]
in the natural order of kernels, and the latter function is identically equal to $1$, and in particular
positive semi-definite.
\end{proof}

The following reformulation of the condition $1/k(T,T^*)\geq 0$ will prove to be convenient. For each multi-index $\alpha\in \bN^d \setminus \{0\}$, consider the polynomial
\begin{equation*}
  \psi_{k,\alpha} = b^{1/2}_{|\alpha|} \binom{|\alpha|}{\alpha}^{1/2} z^\alpha
\end{equation*}
and define the infinite operator tuple
\[
\psi_k(T)=(\psi_{k,\alpha}(T))_{\alpha \in \bN^d \setminus \{0\}}.
\]
It is readily verified that $1/k(T,T^*) \ge 0$ if and only if $\psi_k(T)$ is a row contraction. We emphasize here that this observation crucially uses the fact that the sequence $\{b_n\}_n$ is non-negative, which in turn stems from the complete Nevanlinna-Pick property by Lemma \ref{L:NP-char}. Consequently, we see that if $1/k(T,T^*)\geq 0$ and $\rho$ is a unital completely contractive representation of an operator algebra containing $T_1,\ldots,T_d$, then $1/k(\rho(T),\rho(T)^*)\geq 0$,
where $\rho(T)=(\rho(T_1),\ldots,\rho(T_d))$.

Some basic spectral information is encoded in the inequality $1/k(T,T^*)\geq 0$, as we now show.
As above, we denote by $\sigma(T)\subset \bC^d$ the Taylor spectrum of a commuting tuple of operators $T=(T_1,\ldots,T_d)$ on a Hilbert space.

\begin{lemma}
  \label{L:spectrum}
  Let $\cH$ be a unitarily invariant complete Nevanlinna-Pick
  space on $\bB_d$ with kernel $k$.
  Let $T = (T_1,\ldots,T_d)$ be a commuting tuple of operators on some Hilbert space. If $1/k(T,T^*) \ge 0$, then $\sigma(T) \subset \ol{\bB_d}$.
\end{lemma}

\begin{proof}
Let $\chi$ be a character of the unital norm closed operator algebra generated by $T_1,\ldots,T_d$. We claim that 
\[
\chi(T) = (\chi(T_1),\ldots,\chi(T_d))   \in \ol{\bB_d}.
\]
Since $\chi$ is necessarily completely contractive, the discussion preceding the lemma shows that 
  \[
  1/k(\chi(T),\chi(T)^*) \ge 0.
  \]
  Thus,
  \begin{equation*}
    \sum_{n=1}^\infty b_n ||\chi(T)||^{2 n}
    = \sum_{n=1}^\infty b_n \sum_{|\alpha| = n} \binom{n}{\alpha} |\chi(T)^\alpha|^2 \le 1.
  \end{equation*}
 If $\|\chi(T)\|>1$, then the function $ (1 - \sum_{n=1}^\infty b_n t^n)^{-1}$
 is analytic in a disc centred at the origin of radius strictly greater than $1$. On the other hand, recall that
  \begin{equation*}
    \sum_{n=0}^\infty a_n t^n  = \frac{1}{1 - \sum_{n=1}^\infty b_n t^n}
  \end{equation*}
  and the series on the left-hand side
 is assumed to have radius of convergence $1$ (see Subsection \ref{SS:uispaces}), which is a contradiction. Thus, $\|\chi(T)\|\leq 1$. Finally, we may apply \cite[Proposition 25.3]{muller2007} to conclude that the Taylor spectrum $\sigma(T)$ must lie in $\ol{\bB_d}$.
\end{proof}

If $\cH$ is the Hardy space, then $A(\cH)$ coincides with the disc algebra, which provides
an alternate very concrete description of the elements of $A(\cH)$. In general, such a simple
description is not available. At the very least,
Lemmas \ref{L:1/k} and \ref{L:spectrum} show that $\sigma(M_z^{\cH}) \subset \ol{\bB_d}$,
and thus by the basic properties of the Taylor functional calculus we see that $A(\cH)$ contains
all functions which are analytic
in a neighbourhood of $\ol{\bB_d}$.

We now arrive at a central result of the paper.

\begin{theorem}
  \label{T:A_H_representation}
  Let $\cH$ be a unitarily invariant complete Nevanlinna-Pick space on $\bB_d$
  with reproducing kernel $k$. Let $T=(T_1,\ldots,T_d)$ be a commuting tuple of operators on a Hilbert space $\cE$.
Then, the following assertions are equivalent.
      
\begin{enumerate}[label=\normalfont{(\roman*)}]
\item The tuple $T$ satisfies $1/k(T,T^*)\geq 0$.

\item  There exists a unital completely contractive homomorphism 
\[
\rho: A(\cH) \to B(\cE)
\]
such that $\rho(M_z) = T$. 
 \end{enumerate}
 In this case, the map $\rho$ is necessarily unique and admits a $C^*(A(\cH))$-coextension.
\end{theorem}

\begin{proof}
 By  Lemma \ref{L:1/k} we have that $1/k(M_z,M_z^*) \ge 0$, so the discussion following that same lemma shows that (ii) implies (i).

Assume henceforth that (i) holds so that $1/k(T,T^*)\geq 0$. If a map $\rho$ with the required properties exists, it must be unique since $A(\cH)$ is generated
  by $M_{z_1},\ldots,M_{z_d}$ as a Banach algebra. To establish the existence
  of $\rho$, it suffices to show that there exists a unital completely positive linear map 
  \[
  \psi: C^*(A(\cH)) \to B(\cE)
  \]
  with $\psi(M_{z_j}) = T_j$ and
  $\psi(M_{z_j} M_{z_j}^*) = T_j T_j^*$ for each $j=1,\ldots,d$.
For then Lemma \ref{L:hered_coinvariant} shows
  that if 
  \[
  \pi: C^*(A(\cH)) \to B(\cL)
  \]
  is a Stinespring dilation of $\psi$, then
  $\cE$ is coinvariant for $\pi(A(\cH))$. Therefore, the map $\rho=\psi|_{A(\cH)}$ 
has all the required properties.  It thus remains to construct the map $\psi$. We note that $\sigma(T)\subset\ol{\bB_d}$ by Lemma \ref{L:spectrum}. 
  
  We first
  consider the case where $\sigma(T) \subset \bB_d$. In this setting,
  \cite[Theorems 5 and 6]{ambrozie2002} combined with Remark \ref{rem:AEM} imply
  that there exists an isometry 
  \[
  V: \cE \to \cH \otimes \cE
  \]
  such that $V T_j^* = (M_{z_j}^*\otimes I_\cE) V$ for all $j=1,\ldots,d$.
  As in the proof of the implication (iii) $\Rightarrow$ (ii) of Theorem \ref{T:equiv},
  it follows  that there is a unital $*$-representation $\pi$ of $C^*(A(\cH))$
  such that $\cE$ is coinvariant for $\pi(A(\cH))$
  and 
  \[
  P_{\cE} \pi(M_{z_j}) \big|_{\cE} = T_j  \quad (j=1,\ldots,d).
  \]
  Thus, the compression of $\pi$ to $\cE$ yields the desired
map $\psi$ by virtue of Lemma \ref{L:hered_coinvariant}.

  Finally, we treat the general case where $\sigma(T) \subset \ol{\bB_d}$. For each $0<r<1$ we consider the operator tuple $rT=(rT_1,\ldots,rT_d)$. We have $\sigma(rT)\subset \bB_d$ and it is easy to verify that the assumption $1/k(T,T^*) \ge 0$ implies that $1/k(r T,rT^*) \ge 0$ since $b_n \ge 0$ for all $n \ge 1$.
  By the previous paragraph, for each $0<r<1$ we obtain a unital completely positive map 
  \[
  \psi_r: C^*(A(\cH)) \to B(\cE)
  \]
  such that $\psi_r(M_{z_j}) = r T_j$ and $\psi_r(M_{z_j} M_{z_j}^*) = r^2 T_j T_j^*$ for each $j$. Taking a cluster point of the net $\{\psi_r\}_r$ in an appropriate weak topology (the so-called BW-topology, see \cite[Lemma 7.2]{paulsen2002} for instance) yields the required map $\psi$.
\end{proof}

As an immediate consequence, we can provide a positive answer to Question \ref{Q:intro1}
for all representations of the algebra $A(\cH)$ associated to a unitarily invariant complete Nevanlinna-Pick space on the ball.

\begin{corollary}
  \label{cor:automatic_coextension}
  Let $\cH$ be a unitarily invariant complete Nevanlinna-Pick space on $\bB_d$. Then
  every unital completely contractive representation of $A(\cH)$ admits a $C^*(A(\cH))$-coextension. 
  \qed
\end{corollary}

Example \ref{E:Hardy_Bergman} shows that this corollary may fail without the complete Nevanlinna-Pick property, as the Bergman space $A^2$ is a unitarily invariant space on the disc. Specifically, the implication (ii) $\Rightarrow$ (i)
of Theorem \ref{T:A_H_representation} fails in this case, since $1/k_{A^2}(T,T^*) = 1 - 2 T T^* + T^2 (T^*)^2$ was seen not to be positive if $T$ is the Hardy shift. We will see in Section \ref{sec:NP_nec} that at least for regular spaces, the last corollary
in fact characterizes the complete Nevanlinna-Pick property.

We finish this section by deducing an operator theoretic dilation theorem from
Theorem \ref{T:A_H_representation}.
The first part of this dilation theorem is very much in the spirit
of the work of Agler \cite{agler1982}, Agler and McCarthy \cite{AM00a},
Ambrozie, Engli\v{s} and M\"uller \cite{ambrozie2002} and Arazy and Engli\v{s} \cite{AE2003}.
The Nevanlinna-Pick property
once again guarantees that the existence of dilations and of coextensions is equivalent.
If $\cH = H^2_d$, then the next result is the theorem of M\"uller and Vasilescu \cite{muller1993} and
Arveson \cite{arveson1998}, which was mentioned in the introduction.
Recall from Subsection \ref{ss:regular}
that a unitarily invariant space with kernel $\sum_{n=0}^\infty a_n \langle z,w \rangle^n$ is regular if $\lim_{n \to \infty} a_n / a_{n+1} = 1$.

\begin{theorem}
  \label{T:1/k_dilation}
  Let $\cH$ be a regular unitarily invariant complete Nevanlinna-Pick space on $\bB_d$
  with reproducing kernel $k$ and let $T$ be a commuting
  tuple of operators on a Hilbert space $\cE$. Then the following assertions
  are equivalent.
  \begin{enumerate}[label=\normalfont{(\roman*)}]
\item The tuple $T$ satisfies $1/k(T,T^*) \ge 0$.
    \item
 The tuple $T$ coextends
  to $(M_z^{\cH})^{(\kappa)} \oplus U$ for some cardinal $\kappa$
  and a spherical unitary $U$.
\item
  The tuple $T$ dilates
  to $(M_z^{\cH})^{(\kappa)} \oplus U$ for some cardinal $\kappa$
  and a spherical unitary $U$.
  \end{enumerate}
\end{theorem}

\begin{proof}
 If (i) holds, then Theorem
  \ref{T:A_H_representation} yields a $*$-representation $\pi: C^*(A(\cH)) \to B(\cL)$
  such that $\cE$ is co-invariant for $\pi(A(\cH))$ and $P_{\cK} \pi(M^{\cH}_z) \big|_{\cK} = T$.
  By Lemma \ref{lem:rep_split} and Theorem \ref{T:essentiallynormalan},
  the representation $\pi$ splits as $\pi = \pi_1 \oplus \pi_2$,
  where $\pi_1(M^{\cH}_z)$ is unitarily equivalent to $(M^{\cH}_z)^{(\kappa)}$ for some cardinal $\kappa$
  and $U = \pi_2(M^{\cH}_z)$ is a spherical unitary. Thus (ii) follows.

It is trivial that (ii) implies (iii). 

Finally, assume that (iii) holds. Observe that the spherical unitary $U$ gives rise
  to a $*$-representation of $C(\partial \bB_d)$ which maps $z$ to $U$, hence
  Theorem \ref{T:essentiallynormalan} implies that
  there exists a $*$-representation $\pi$ of $C^*(A(\cH))$ which maps $M^{\cH}_z$
  to $(M^{\cH}_z)^{(\kappa)} \oplus U$. The compression of $\pi$ to $\cE$ is then a unital completely
  contractive representation of $A(\cH)$ which maps $M^{\cH}_z$ to $T$, so that
  $1/k(T,T^*) \ge 0$ by Theorem \ref{T:A_H_representation}.
\end{proof}

\section{Boundary representations and hyperrigidity for unitarily invariant spaces}\label{S:AH}
In this section, we investigate the notions of boundary representations, $C^*$-envelope and hyperrigidity
(see Subsection \ref{S:prelimuep}) for the algebras $A(\cH)$.
In particular, we answer Question \ref{Q:intro2} for these algebras whenever $\cH$ is regular.

It is not hard to see that every spherical unitary tuple gives rise to a representation of $A(\cH)$.
We show below via an application of Theorem \ref{T:sphunitary_uep} that such a representation
has the unique extension property provided that the reproducing kernel is unbounded.

\begin{proposition}
  \label{P:UEP_bigspace}
  Let $\cH$ be a unitarily invariant complete Nevanlinna-Pick space on $\bB_d$
  with reproducing kernel 
  \[
  k(z,w) = \sum_{n = 0}^\infty a_n \langle z,w \rangle^n
  \]
  and let $U = (U_1,\ldots,U_d)$ be a spherical unitary on a Hilbert space $\cE$.
  Then, there exists a unique unital completely contractive homomorphism  $\rho: A(\cH) \to B(\cE) $ such that $\rho(M_{z})=U.$  If $\sum_{n=0}^\infty a_n = \infty$, then $\rho$ has the unique extension property.
\end{proposition}

\begin{proof}
  Uniqueness of $\rho$ follows from the fact that $A(\cH)$ is generated
  by $M_{z_1},\ldots,M_{z_d}$ as a Banach algebra. To prove existence,
  recall that the complete Nevanlinna-Pick property implies that there exists
  a sequence $\{b_n\}_n$ of non-negative numbers such that
  \[
  1-\frac{1}{k(z,w)}=\sum_{n=1}^\infty b_n \langle z,w\rangle^n \quad (z,w\in \bB_d).
  \]
  Note that $\sum_{n=1}^\infty b_n \le 1$, and $\sum_{n=1}^\infty b_n = 1$ if and only
  if $\sum_{n=0}^\infty a_n = \infty$.
  Furthermore, recall from the discussion following Lemma \ref{L:1/k}
that for each multi-index $\alpha\in \bN^d\setminus \{0\}$, we defined
\[
\psi_{k,\alpha}(z)=b_{|\alpha|}^{1/2}\binom{|\alpha|}{\alpha}^{1/2}z^\alpha.
\]
By the Putnam-Fuglede theorem, each $\psi_{k,\alpha}(U)$ is normal and
  \begin{align*}
    \sum_{\alpha \in \bN^d \setminus \{0\}} \psi_{k,\alpha}(U)\psi_{k,\alpha}(U)^*
    = \sum_{\alpha \in \bN^d \setminus \{0\}} b_{|\alpha|} \binom{|\alpha|}{\alpha}
    U^\alpha (U^*)^\alpha
    = \sum_{n=1}^\infty b_n I.
  \end{align*}
  Thus,  the tuple $\psi_k(U)=(\psi_{k,\alpha}(U))_\alpha$ is a row contraction and hence $1/k(U,U^*) \ge 0$, so that
  an application of
  Theorem \ref{T:A_H_representation} yields the existence of $\rho$.
  
If $\sum_{n = 1}^\infty b_n = 1$, then $\psi_k(U)$ is a spherical unitary. On the other hand, by Lemma \ref{L:1/k} we know that $\psi_k(M_z)$ is a row contraction. Since $A(\cH)$ coincides with the unital norm closed algebra generated by the operators $\psi_{k,\alpha}(M_z)$ for $\alpha\in \bN^d\setminus \{0\}$,   the representation $\rho$ has
  the unique extension property by Theorem \ref{T:sphunitary_uep}.
 \end{proof}

  We remark that the mere existence of $\rho$ in the preceding result also follows
  from the elementary fact that the norm on $A(\cH)$ dominates the supremum norm over the unit sphere,
  so that the restriction map $A(\cH) \to C(\partial \bB_d)$ is completely contractive.

We are now in the position to identify all the boundary representations for $A(\cH)$ when
$\cH$ is regular.
Observe that Theorem \ref{T:essentiallynormalan} shows in particular that if $\cH$ is a regular unitarily
invariant space, then every $\zeta\in \partial \bB_d$ gives rise to a character
$\rho_\zeta$ on $C^*(A(\cH))$ of evaluation at $\zeta$ which annihilates the compact operators. 

\begin{theorem}\label{T:bdryan}
Let $\cH$ be a regular unitarily invariant complete Nevanlinna-Pick space on $\bB_d$ with reproducing kernel 
\[
k(z,w)=\sum_{n=0}^\infty a_n \langle z,w \rangle^n.
\]
Suppose that $\cH$ is not the Hardy space $H^2(\bD)$.
  \begin{enumerate}[label=\normalfont{(\alph*)}]
    \item 
    If $\sum_{n=0}^\infty a_n = \infty$,
then the boundary representations for $A(\cH)$ are precisely the characters $\rho_\zeta$ for $\zeta \in \partial \bB_d$
      and the identity representation of $C^*(A(\cH))$.
    \item
    If 
      $\sum_{n=0}^\infty a_n < \infty$,
then the identity representation of $C^*(A(\cH))$
      is the only boundary representation  for $A(\cH)$.
  \end{enumerate}
\end{theorem}

\begin{proof}
  By Theorem \ref{T:essentiallynormalan}, we have that $\cK(\cH)\subset C^*(A(\cH))$ and $C^*(A(\cH)) / \cK(\cH)$ is $*$-isomorphic to $C( \partial \bB_d)$.
  Lemma \ref{lem:rep_split} shows that the only possible irreducible $*$-representations of $C^*(A(\cH))$
  are therefore the identity representation and the point evaluation characters $\rho_\zeta$ for $\zeta \in \partial \bB_d$.

We first show that the identity representation is always a boundary representation for $A(\cH)$.
 By Arveson's boundary theorem \cite[Theorem 2.1.1]{arveson1972}, it suffices to show that
the quotient map
\[
q:C^*(A(\cH))\to C^*(A(\cH))/\cK(\cH)\cong C(\partial \bB^d)
\]
is not completely isometric on $A(\cH)$. 
We will show that it is not even isometric. First note that $q(M_f)=f$ whence $\|q(M_f)\|=\|f\|_\infty$ for every polynomial $f$.
 Next, it
follows from \cite[Proposition 6.4]{hartz2015isom} that \begin{equation*}
  ||M_{z_1^n}||^2 = ||z_1^n||^2_{\cH} = 1/a_n,
\end{equation*}
whereas $||z_1^n||_{\infty} = 1$.
Thus, $q$ is not isometric unless $a_n = 1$ for all $n\geq 0$.
But this would imply $\cH$ is the Drury-Arveson space $H^2_d$,
and it is well known then that $q$ is not isometric if $d \ge 2$ (see \cite[Section 3.8]{arveson1998}). Indeed,
\begin{equation*}
  ||M^{H^2_d}_{z_1 z_2}||^2 = ||z_1 z_2||^2_{H^2_d} = 1/2,
\end{equation*}
whereas $||z_1 z_2||^2_\infty = 1/4$.
Thus, in order for $q$ to be isometric we must have $d=1$ and $\cH =H^2_1= H^2(\bD)$ contrary to our assumption. We conclude that the identity representation of $C^*(A(\cH))$ is a boundary representation for $A(\cH)$.

If $\sum_{n=0}^\infty a_n = \infty$, then Proposition \ref{P:UEP_bigspace} shows that 
the character $\rho_\zeta$ of $C^*(A(\cH))$ is a boundary representation for $A(\cH)$ for every $\zeta\in \partial \bB_d$. 

Finally, assume that $\sum_{n=0}^\infty a_n < \infty$. We must show that no point evaluation character $\rho_\zeta$ of $C^*(A(\cH))$ is a boundary representation for $A(\cH)$.
  In this case, $\cH$ can be identified with a reproducing kernel Hilbert space on $\ol{\bB_d}$
  in a natural way (see Subsection \ref{SS:an}). Let $\zeta \in \partial \bB_d$. Then
  the restriction  $\rho_{\zeta} |_{A(\cH)}$ is a unital completely contractive functional which admits two extensions to a state on $C^*(A(\cH))$, namely
  $\rho_{\zeta}$ itself and the state $\varphi$ defined by
  \begin{equation*}
    \varphi(a) =
    \frac{\langle a k_{\zeta}, k_{\zeta} \rangle}{\|k_{\zeta}\|^2}, \quad a\in C^*(A(\cH))
  \end{equation*}
  where $k_\zeta=k(\cdot,\zeta)\in \cH$.
  Moreover, these two states are distinct as $\rho_{\zeta}$ annihilates the compacts whereas
  $\varphi$ clearly does not. We conclude that $\rho_{\zeta}$ is not a boundary representation for $A(\cH)$. 
\end{proof}

We remark that the question of whether the identity representation is a boundary representation
in the context of unitarily invariant spaces
was already considered by Guo, Hu and Xu \cite{GHX04}, and a proof of that fact in the setting
of Theorem \ref{T:bdryan} could also be based on \cite[Proposition 4.9]{GHX04}.

Our methods also allow us to determine when the algebras $A(\cH)$ are hyperrigid.
Recall that a unital operator algebra $\cA$ is said to be hyperrigid inside of $C^*(\cA)$
if the restriction of every unital $*$-representation of $C^*(\cA)$ to $\cA$
has the unique extension property.

\begin{theorem}
  \label{T:hyperrigid}
  Let $\cH$ be a regular unitarily invariant complete Nevanlinna-Pick space on $\bB_d$ with reproducing kernel
  \begin{equation*}
    k(z,w) = \sum_{n=0}^\infty a_n \langle z,w \rangle^n.
  \end{equation*}
  Suppose that $\cH$ is not the Hardy space $H^2(\bD)$.
  \begin{enumerate}[label=\normalfont{(\alph*)}]
    \item If $\sum_{n=0}^\infty a_n = \infty$, then $A(\cH)$ is hyperrigid inside of $C^*(A(\cH))$.
    \item If $\sum_{n=0}^\infty a_n < \infty$, then $A(\cH)$ is not hyperrigid inside of $C^*(A(\cH))$.
  \end{enumerate}
\end{theorem}

\begin{proof}
(a)  By Theorem \ref{T:essentiallynormalan} and
  Lemma \ref{lem:rep_split}, every unital $*$-representation $\pi$ of $C^*(A(\cH))$ decomposes
  as a direct sum $\pi_1 \oplus \pi_2$, where $\pi_1$ is unitarily equivalent
  to a multiple of the identity representation and $\pi_2$ annihilates the ideal of compact operators.
  Since direct sums of representations with the unique extension property have the unique
  extension property themselves \cite[Proposition 4.4]{Arveson11}, it suffices
  to show that the restriction to $A(\cH)$ of the identity representation and of every representation
  which annihilates the compacts has the unique extension property.
  The assertion about the identity representation is contained in part (a) of Theorem \ref{T:bdryan}. Finally,
  if $\pi$ is a representation of $C^*(A(\cH))$ which annihilates
  the compacts, then an application of Theorem \ref{T:essentiallynormalan} shows
  that $(\pi(M_{z_1}), \ldots, \pi(M_{z_d}))$ is a spherical unitary, so that $\pi|_{A(\cH)}$ has the unique
  extension property by
  Proposition \ref{P:UEP_bigspace}.
  
  (b) By part (b) of Theorem \ref{T:bdryan}, the restrictions $\rho_\zeta \big|_{A(\cH)}$ of the point evaluation characters of $C^*(A(\cH))$
  do not have the unique extension property.
\end{proof}

Arveson conjectured that a separable unital operator algebra $\cA$ is hyperrigid
if and only if every irreducible representation of $C^*(\cA)$ is a boundary representation
for $\cA$ \cite[Conjecture 4.3]{Arveson11}. Combining Theorem \ref{T:bdryan} and Theorem
\ref{T:hyperrigid}, we see that for regular unitarily invariant complete Nevanlinna-Pick spaces on the ball, the algebras $A(\cH)$
support Arveson's conjecture.

Using our knowledge about the boundary representations of $A(\cH)$,
we can also easily identify the $C^*$-envelope.

\begin{corollary}\label{C:envelope}
Let $\cH$ be a regular unitarily invariant complete Nevanlinna-Pick space on the unit ball. If $\cH$ is the Hardy space $H^2(\bD)$ then $C^*_e(A(\cH))=C(\bT)$, while if it is not, then $C^*_e(A(\cH))=C^*(A(\cH))$.
\end{corollary}
\begin{proof}
The case where $\cH=H^2(\bD)$ is classical: $A(H^2(\bD))$ is the disc algebra, whose Shilov boundary is the unit circle $\bT$.  The alternative is a straightforward consequence of the fact that the identity representation is a boundary representation by Theorem \ref{T:bdryan} (see the discussion in Subsection \ref{S:prelimuep}).
\end{proof}

We finish this section by observing that the $C^*$-algebras $C^*(A(\cH))$ are in fact isomorphic to one another for every regular unitarily invariant 
space on the ball. For some classical spaces, this fact
is mentioned in \cite[Section 5]{arveson1998}. We provide a proof
below for the reader's convenience.

\begin{proposition}\label{P:toeplitzisom}
  Let $\cH$ be a regular unitarily invariant space on $\bB_d$.
  Then, $C^*(A(\cH))$ and $C^*(A(H^2_d))$ are unitarily equivalent.
  More precisely, there exists
  a unitary $U: H^2_d \to \cH$ such that the operators
  \begin{equation*}
    U^* M_{z_j}^{\cH} U - M_{z_j}^{H^2_d} \quad (j=1,\ldots,d)
  \end{equation*}
  are compact, and $U^* C^*(A(\cH)) U=C^*(A(H^2_d)).$
\end{proposition}

\begin{proof}
There exists a unique unitary operator 
  \[
  U: H^2_d \to \cH
  \]
  which satisfies
  \begin{equation*}
    U p = \sqrt{a_n} p
  \end{equation*}
  for every homogeneous polynomial $p$ of degree $n$ and every $n\geq 0$. A simple computation shows that
  if $p$ is a homogeneous polynomial of degree $n$, then
  \begin{equation*}
    (U^* M_{z_j}^{\cH} U - M_{z_j}^{H^2_d}) p = 
    (\sqrt{a_n / a_{n+1}} -1) z_j p
  \end{equation*}
  for $j=1,\ldots,d$.
By regularity of $\cH$, we know that
  \[
  \lim_{n\to \infty}a_n/a_{n+1}=1
  \]
  whence we infer that $U^* M_{z_j}^{\cH} U - M_{z_j}^{H^2_d}$ can be approximated in norm by finite rank operators, and thus is compact for every $j=1,\ldots,d$.
  Hence, the map $T \mapsto U^* T U$
  sends $A(\cH)$ into $A(H^2_d)+\cK(H^2_d)\subset C^*(A(H^2_d))$. Clearly, this implies that it sends $C^*(A(\cH))$ into $C^*(A(H^2_d))$. By a similar argument,
  the map $X \mapsto U X U^*$ sends $C^*(A(H^2_d))$ into $C^*(A(\cH))$, and the proof is complete.
  \end{proof}

By Corollary \ref{C:envelope}, we conclude that there are only two possible isomorphism
classes of $C^*$-envelopes for the algebras $A(\cH)$ in this setting.

\section{Homogeneous ideals and Arveson's conjecture}
\label{S:ess_norm}

In Section \ref{S:ideals}, we will extend the results of Sections \ref{S:coextue} and
\ref{S:AH} to quotients of the algebras $A(\cH)$ by homogeneous ideals.
For the moment, we collect here the necessary preliminaries.

Let $\cI \subset \bC[z_1,\ldots,z_d]$
be a proper homogeneous ideal and let $\cH$ be a unitarily invariant space on $\bB_d$
such that the polynomials are multipliers of $\cH$.
We define $\cH_\cI = \cH \ominus \cI$, which is coinvariant for $A(\cH)$ (and indeed for $\Mult(\cH)$),
and we put
\[
S^{\cI}_j=P_{\cH_\cI}M_{z_j}|_{\cH_\cI} \quad (1\leq j \leq d).
\]
Then $S^\cI=(S_1^\cI,\ldots,S_d^\cI)$ is a commuting tuple of operators on $\cH_\cI$ and
it is immediate that $p(S^{\cI})=0$ for every $p\in \cI$.
When the ideal $\cI$ is apparent from context, we may suppress it from the notation and write simply $S$ instead of $S^\cI$.
We now define $A(\cH_\cI)$ as the norm closed algebra generated by $S_1^\cI,\ldots,S_d^\cI$.
The previously studied setting of unitarily invariant spaces on $\bB_d$ corresponds to $\cI=\{0\}$.
We will see in Corollary \ref{C:quotient} that $A(\cH_\cI)$ may alternatively be described
as the quotient of $A(\cH)$ by the norm closure of $\cI$. Part of our motivation for studying these algebras is that they play a role in the recent study of operator theory
on varieties \cite{DRS2011,DRS2015,hartz2015isom} and, as we will see, they are connected to Arveson's essential normality conjecture. 

Let
\begin{equation*}
  V(\cI) = \{ \lambda \in \bC^d: p(\lambda) = 0 \text{ for all } p \in \cI \}
\end{equation*}
denote the vanishing locus of $\cI$.
If the ideal $\cI$ is radical, then the space $\cH_\cI$ can be regarded as a reproducing kernel
Hilbert space on $V(\cI) \cap \bB_d$.
Indeed,
using Hilbert's Nullstellensatz, one can show that the restriction map induces
a unitary operator between $\cH_\cI$ and $\cH \big|_{V(\cI) \cap \bB_d}$ (see
\cite[Lemma 7.4]{hartz2015isom}). Correspondingly, the algebra $A(\cH_\cI)$ is identified
with the norm closure of the polynomials in $\Mult(\cH \big|_{V(\cI) \cap \bB_d})$. In particular,
$A(\cH_\cI)$ is an algebra of multipliers, so it fits within the scope of the paper. On the
other hand, if $\cI$ is not a radical ideal, then $A(\cH_\cI)$ contains non-zero nilpotent elements,
so it is not semi-simple and in particular not an algebra of functions. Nevertheless,
we feel that it is worthwhile to include this more general case, since it does not present
any additional difficulties and has been studied as well (see, for example, \cite{arveson2007,DRS2011}).

While the early results in Section \ref{S:coextue} generalize in a straightforward way, Theorem \ref{T:1/k_dilation}
and most of the results in Section \ref{S:AH} depend on Theorem \ref{T:essentiallynormalan}. Generalizing
this theorem is very difficult. Indeed, one of its assertions
is that the tuple $M_z$ on $\cH$ is \emph{essentially normal}: each operator $M_{z_j}$ is normal modulo the ideal of compact operators $\cK(\cH)$. In the case of the Drury-Arveson space,
the question of whether $S^{\cI}$ is essentially normal is Arveson's famous essential
normality conjecture \cite{arveson2002,arveson2005}. Even though this conjecture
has witnessed exciting progress recently (see for example \cite{EE2015} and \cite{DTY14}),
it remains open in general. We do not address this conjecture directly and instead
make the following definition.
We say that a homogeneous ideal $\cI \subset \bC[z_1,\ldots,z_d]$
\emph{satisfies Arveson's conjecture} if the commuting tuple given by
\[
P_{H^2_d\ominus \cI}M_{z_j}\Big|_{H^2_d\ominus \cI} \quad (1\leq j \leq d)
\]
is essentially normal.

We will require the following basic property of $C^*(A(\cH))$, whose proof is an adaptation of the proof of
\cite[Proposition 2.5]{arveson2007}.

\begin{lemma}\label{L:irredideal}
  Let $\cH$ be a regular unitarily invariant space on $\bB_d$.
  If $\cI \subset \bC[z_1,\ldots,z_d]$
  is a proper homogeneous ideal, then the algebra $A(\cH_\cI)$ is irreducible, and indeed,
  $C^*(A(\cH_\cI))$ contains the ideal of compact operators on $\cH_{\cI}$.
\end{lemma}
\begin{proof}
Let $P\in B(\cH_\cI)$ be a projection commuting with $S_1,\ldots,S_d$. In particular, if we denote by $Q$ the projection onto the closed subspace $\vee_{j=1}^d S_j \cH_\cI$
then $QP=PQ$. Hence $(I-Q)P=P(I-Q)$. Since the polynomials are dense in $\cH$, we see that $(I-Q)\cH_\cI$ is the one-dimensional subspace spanned by the constant function $e=1$, which belongs to $\cH_\cI$ since $\cI$ is homogeneous and proper. Thus, $P$ commutes with the rank one projection $e\otimes e$, so that either $Pe=e$ or $Pe=0$. Given a polynomial $f$, we have
\[
PP_{\cH_\cI}f=Pf(S)e=f(S)Pe.
\]
Using once again that the polynomials are dense in $\cH$, we see that either $P=I$ or $P=0$. Thus, $A(\cH_\cI)$ is irreducible.

Finally, the operator $I - \sum_{j=1}^d M_{z_j} M_{z_j}^*$
is compact according to Theorem \ref{T:essentiallynormalan}. From coinvariance of $\cH_{\cI}$,
we deduce that
\[
I_{\cH_\cI}-\sum_{j=1}^d S_j S_j^* =P_{\cH_\cI}\left( I-\sum_{j=1}^d M_{z_j}M_{z_j}^*\right)P_{\cH_\cI}
\]
is compact as well. Moreover, this operator is not zero, as
\[
\left(I_{\cH_\cI}-\sum_{j=1}^d S_j S_j^*\right) e=e.
\]
Thus, $C^*(A(\cH_\cI))$ contains a non-zero compact operator, and hence all compact operators
on $\cH_\cI$ since this $C^*$-algebra is irreducible.
\end{proof}

We observe that by the Putnam-Fuglede theorem, $\cI$ satisfies Arveson's conjecture if
and only if
$C^*(A(H^2_d \ominus \cI))) / \cK(H^2_d\ominus \cI)$ is a commutative $C^*$-algebra.
Although this property appears at first glance to be specific to the Drury-Arveson space,
it is in fact possible to replace $H^2_d$ with any regular unitarily invariant space.
This was shown by Wernet in his PhD thesis
\cite[Corollary 2.49]{Wernet14}.
More generally, the following extension of Proposition \ref{P:toeplitzisom} holds.

\begin{proposition}
  \label{P:toeplitzisomideal}
  Let $\cH$ be a regular unitarily invariant space on $\bB_d$ and
  let $\cI \subset \bC[z_1,\ldots,z_d]$ be a proper homogeneous ideal.
  Then, the $C^*$-algebras $C^*(A(\cH_{\cI}))$ and $C^*(A(H^2_d \ominus \cI))$ are unitarily equivalent.
  In fact, there exists
  a unitary $U: H^2_d \ominus \cI \to \cH_{\cI}$ such that the operators
  \begin{equation*}
    U^* P_{\cH_{\cI}} M_{z_j}^{\cH} \big|_{\cH_{\cI}} U - P_{H^2_d \ominus \cI} M_{z_j}^{H^2_d}\big|_{H^2_d \ominus \cI} \quad (j=1,\ldots, d)
  \end{equation*}
  are compact, and 
  \[
  U^* C^*(A(\cH_\cI)) U=C^*(A(H^2_d\ominus \cI)).
  \]
  In particular,  $P_{\cH_{\cI}} M_z^{\cH} \big|_{\cH_{\cI}}$
  is essentially normal if and only if $P_{H^2_d \ominus \cI} M_z^{H^2_d} \big|_{H^2_d \ominus \cI}$
  is essentially normal.
\end{proposition}

\begin{proof}
  Since $\cI$ is homogeneous, the unitary operator $U: H^2_d \to \cH$ constructed in the proof of Proposition
  \ref{P:toeplitzisom} maps $\cI$ onto $\cI$. Therefore, it maps $H^2_d \ominus \cI$ onto $\cH_{\cI}$
  and
  \begin{equation*}
    U P_{H^2_d \ominus \cI} = P_{\cH_{\cI}} U,
  \end{equation*}
  so that for $j=1,\ldots,d$, we obtain
  \begin{equation*}
    U^* P_{\cH_{\cI}} M_{z_j}^{\cH} \big|_{\cH_{\cI}} U - 
    P_{H^2_d \ominus \cI} M_{z_j}^{H^2_d} \big|_{H^2_d \ominus \cI}
    = P_{H^2_d \ominus \cI} (U^* M_{z_j}^{\cH} U - M_{z_j}^{H^2_d}) \big|_{H^2_d \ominus \cI},
  \end{equation*}
  which is a compact operator by Proposition \ref{P:toeplitzisom}.
  Since both   $C^*(A( H^2_d \ominus \cI))$ and $C^*(A(\cH_\cI))$ contain the compact operators
  by Lemma \ref{L:irredideal}, we see as in the proof of Proposition \ref{P:toeplitzisom}
  that   
  \[
  U^* C^*(A(\cH_\cI)) U=C^*(A(H^2_d\ominus \cI)). \qedhere
  \]
\end{proof}

We remark that Wernet in fact obtained finer results relating the $p$-essential
normality of operator tuples $P_{\cH_{\cI}} M_z \big|_{\cH_{\cI}}$
in different spaces $\cH$, but the version above suffices for our purposes.
Whenever Arveson's conjecture holds, we obtain the following generalization of Theorem \ref{T:essentiallynormalan}.

\begin{theorem}\label{T:SESideal}
Let $\cH$ be a regular unitarily invariant space on $\bB_d$
and let $\cI\subset \bC[z_1,\ldots,z_d]$ be a proper
homogeneous ideal satisfying Arveson's conjecture.
Then there exists a short exact sequence
\begin{equation*}
  0 \longrightarrow \cK(\cH_{\cI}) \longrightarrow C^*(A(\cH_{\cI})) \longrightarrow C(V(\cI) \cap \partial \bB_d) \longrightarrow 0,
\end{equation*}
where the first map is the inclusion and the second map sends $S_j$ to the coordinate
function $z_j$ for $1 \le j \le d$.
\end{theorem}
\begin{proof}
By  \cite[Theorem 5.1 (4)]{GW08} (and the following discussion therein) there is a short exact sequence
\begin{equation*}
  0 \longrightarrow \cK(H^2_d \ominus \cI) \longrightarrow C^*(A(H^2_d \ominus \cI)) \longrightarrow C(V(\cI) \cap \partial \bB_d) \longrightarrow 0.
\end{equation*}
Using this fact along with Proposition \ref{P:toeplitzisomideal} completes the proof.
\end{proof}

We use the convention that $C(\emptyset)$ is
the zero space. Thus, under the conditions of the previous theorem, if $V(\cI) \cap \partial \bB_d = \emptyset$ then $C^*(A(\cH_\cI)) =
\cK(\cH_\cI)$.

\section{Algebras of multipliers and homogeneous ideals}\label{S:ideals}
We now turn our efforts to extending the results of Sections \ref{S:coextue} and \ref{S:AH}
to the algebras $A(\cH_\cI)$ introduced above.
While we will seemingly be retreading old ground by proving  generalizations of previously stated results,
we feel that dealing with the more general setting separately makes for easier reading,
in particular given the required preparation (which was taken care of in Section \ref{S:ess_norm}).

Our first goal is to generalize Theorem \ref{T:A_H_representation}.
To this end, we begin
with the following straightforward adaptation of \cite[Theorems 5 and 6]{ambrozie2002}.

\begin{theorem}
  \label{T:AEM_ideal}
  Let $\cH$ be a unitarily invariant complete Nevanlinna-Pick space on $\bB_d$
  with reproducing kernel $k$. Let $\cI\subset \bC[z_1,\ldots,z_d]$ be a proper homogeneous ideal. Let $T = (T_1,\ldots,T_d) $
  be a commuting tuple of operators on some Hilbert space $\cE$ with the following properties:
  \begin{enumerate}[label=\normalfont{(\arabic*)}]
    \item $\sigma(T) \subset \bB_d$,
    \item $1/k(T,T^*) \ge 0$, and
    \item $p(T) = 0$ for all $p \in \cI$.
  \end{enumerate}
  Then there exists an isometry $V: \cE \to \cH_\cI \otimes \cE$
  such that
  \begin{equation*}
    (S_j^* \otimes I_{\cE}) V = V T_j^*
  \end{equation*}
  for $j = 1,\ldots,d$.
\end{theorem}

\begin{proof}
  In view of conditions (1) and (2), we may invoke \cite[Theorems 5 and 6]{ambrozie2002} to find an isometry $V: \cE \to \cH \otimes \cE$ such that
  $(M_{z_j}^* \otimes I_{\cE}) V = V T_j^*$ for all $j=1,\ldots,d$.  Hence
  \begin{equation*}
    (M_{p}^* \otimes I_{\cE}) V = V p(T)^*
  \end{equation*}
  for every $p \in \bC[z_1,\ldots,z_d]$. We claim that the range
  of $V$ is contained in $\cH_\cI \otimes \cE$.
  To this end, let $p \in \cI$ and $x,y \in \cE$. By (3) we find $p(T)=0$ and thus
  \begin{equation*}
    \langle V x, p \otimes y \rangle =
    \langle (M_p^* \otimes I_{\cE}) V x, 1 \otimes y \rangle
    = \langle V p(T)^* x, 1 \otimes y \rangle = 0,
  \end{equation*}
which establishes the claim. Therefore, we may regard $V$
as an isometry from $\cE$ into $\cH_{\cI} \otimes \cE$, and
since $\cH_\cI \otimes \cE$
  is invariant under $M_{z_j}^* \otimes I_{\cE}$, we see that
  \begin{equation*}
    (S_{j}^* \otimes I_{\cE}) V = V T_j^*
  \end{equation*}
  for $j = 1,\ldots,d$.
\end{proof}

It is noteworthy that the spectral assumption (1) in the previous theorem can be replaced with a suitable
purity condition. This is achieved by invoking
\cite[Corollary 3.2]{AE2003} rather than \cite[Theorems 5 and 6]{ambrozie2002}.
However, the above version suffices for our purpose, which is to
establish the following generalization of Theorem \ref{T:A_H_representation} and Corollary \ref{cor:automatic_coextension}. 

\begin{theorem}
  \label{T:rep_A_I}
  
   Let $\cH$ be a unitarily invariant complete Nevanlinna-Pick space on $\bB_d$
  with reproducing kernel $k$. Let $\cI \subset \bC[z_1,\ldots,z_d]$ be a proper homogeneous ideal. Let $T=(T_1,\ldots,T_d)$ be a commuting tuple of operators on a Hilbert space $\cE$.
Then, the following assertions are equivalent.
      
\begin{enumerate}[label=\normalfont{(\roman*)}]
\item The tuple $T$ satisfies $1/k(T,T^*)\geq 0$ and $p(T)=0$ for every $p\in \cI$.

\item There exists a unique unital completely contractive homomorphism 
\[
\rho: A(\cH_\cI) \to B(\cE)
\]
such that $\rho(S) = T$. 
 \end{enumerate}
 Moreover, in this case the map $\rho$ admits a $C^*(A(\cH_\cI))$-coextension.
 In particular, every unital completely contractive representation of $A(\cH_\cI)$
 admits a $C^*(A(\cH_\cI))$-coextension.
\end{theorem}

\begin{proof}
Assume first that (ii) holds. The composition of the map
\begin{equation*}
  A(\cH) \to A(\cH_\cI), \quad M_\varphi \mapsto P_{\cH_{\cI}} M_\varphi \big|_{\cH_{\cI}},
\end{equation*}
with $\rho$ is a unital completely contractive representation of $A(\cH)$ which maps $M_{z}$ to
$T$, hence $1/k(T,T^*) \ge 0$ by Theorem \ref{T:A_H_representation}.
Moreover, $p(T)=\rho(p(S))=0$ for every $p\in \cI$, and (i) follows.

The proof that (i) implies (ii) and of the fact that $\rho$ admits a $C^*(A(\cH_\cI))$-coextension is identical to that of the corresponding implication in Theorem \ref{T:A_H_representation}, upon using Theorem \ref{T:AEM_ideal}. Note that the fact that $\cI$ is homogeneous is needed to ensure that $p(r T) = 0$ for all $p \in \cI$ and all $0<r<1$.
\end{proof}

We now make a brief digression to illustrate how the last result can be used to show that $A(\cH_\cI)$ is completely
isometrically isomorphic to $A(\cH) / \ol{\cI}$, where $\ol{\cI}\subset A(\cH)$ is the closure of $\cI$ in the operator norm. 
\begin{corollary}
  \label{C:quotient}
  Let $\cH$ be a unitarily invariant complete Nevanlinna-Pick space on $\bB_d$
  and let $\cI \subset \bC[z_1,\ldots,z_d]$ be a proper homogeneous ideal. Then the map
  \begin{equation*}
    \Phi: A(\cH) / \ol{\cI} \to A(\cH_\cI), \quad M_\varphi + \ol{\cI} \mapsto P_{\cH_{\cI}} M_\varphi \big|_{\cH_{\cI}},
  \end{equation*}
  is a unital completely isometric isomorphism.
\end{corollary}

\begin{proof}
  Since $\cI$ is contained in the kernel of the unital completely contractive homomorphism
  \begin{equation*}
    A(\cH) \to A(\cH_\cI), \quad M_\varphi \mapsto P_{\cH_{\cI}} M_\varphi \big|_{\cH_{\cI}},
  \end{equation*}
  the map $\Phi$ is
  a well defined unital completely contractive homomorphism.

By the Blecher-Ruan-Sinclair theorem \cite[Corollary 16.7]{paulsen2002}, there exists
  a unital completely isometric homomorphism $\pi: A(\cH) / \ol{\cI} \to B(\cE)$
  for some Hilbert space $\cE$. Let $q: A(\cH) \to A(\cH) / \ol{\cI}$ denote the
  unital completely contractive quotient map, and let $T = \pi( q(M_z))$. Then it is obvious
  that $p(T) = 0$ for all $p \in \cI$,
  and applying Theorem \ref{T:A_H_representation} to $\pi \circ q$ shows that $1/k(T,T^*) \ge 0$. Next, Theorem \ref{T:rep_A_I}
  yields a unital completely contractive homomorphism $\rho: A(\cH_\cI) \to B(\cE)$
  which maps $S^\cI$ to $T$. Observe that the range of $\rho$ is contained in the range of $\pi$,
  and that $\pi^{-1} \circ \rho$ is a unital completely contractive homomorphism
  which maps $S^{\cI}$ to $M_z + \ol{\cI}$. Thus, $\pi^{-1} \circ \rho$ is a completely contractive
  inverse of $\Phi$, so that $\Phi$ is a unital completely isometric isomorphism.
\end{proof}

Alternatively, a proof of this corollary can be based on the commuting lifting theorem
for complete Nevanlinna-Pick spaces of Ball, Trent and Vinnikov \cite{BTV01}.

Next, we extract a generalization of Theorem \ref{T:1/k_dilation} from Theorem \ref{T:rep_A_I}. In order to apply Theorem \ref{T:SESideal}, we need to require that $\cI$
satisfies Arveson's conjecture.

\begin{theorem}
  \label{T:1/k_dilation_ideal}
  Let $\cH$ be a regular unitarily invariant complete Nevanlinna-Pick space on $\bB_d$
  with reproducing kernel $k$, let $\cI \subset \bC[z_1,\ldots,z_d]$
  be a proper homogeneous ideal which satisfies Arveson's conjecture and let $T$ be a commuting
  tuple of operators on a Hilbert space . Then the following assertions
  are equivalent.
  \begin{enumerate}[label=\normalfont{(\roman*)}]
  \item The tuple $T$ satisfies $1/k(T,T^*) \ge 0$ and $p(T) = 0$ for all $p \in \cI$.
    \item
  The tuple $T$ coextends
  to $(S^{\cI})^{(\kappa)} \oplus U$ for some cardinal $\kappa$
  and a spherical unitary $U$ whose joint spectrum is contained in $V(\cI) \cap \partial \bB_d$.
\item The tuple $T$ dilates to $(S^\cI)^{(\kappa)} \oplus U$ for some cardinal $\kappa$ and a spherical
  unitary $U$ whose joint spectrum is contained in $V(\cI) \cap \partial \bB_d$.
  \end{enumerate}
\end{theorem}

\begin{proof}
  The proof is almost identical to that of Theorem \ref{T:1/k_dilation}; we merely
  need to replace the application of Theorem \ref{T:A_H_representation} and Theorem
  \ref{T:essentiallynormalan} with an application of Theorem \ref{T:rep_A_I} and
  Theorem \ref{T:SESideal}, respectively.
\end{proof}

We remark that for a spherical unitary $U$, the following assertions are equivalent:
\begin{enumerate}[label=\normalfont{(\roman*)}]
  \item the joint spectrum of $U$ is contained in $V(\cI) \cap \partial \bB_d$,
  \item $p(U) = 0$ for all $p \in \cI$, and
  \item $p(U) = 0$ for all $p \in \sqrt{\cI}$, the radical of $\cI$.
\end{enumerate}
Thus,  Theorem \ref{T:1/k_dilation_ideal} could be reformulated using these conditions where appropriate.

Another consequence of Theorem \ref{T:rep_A_I} is a generalization of Proposition \ref{P:UEP_bigspace}.

\begin{proposition}
  \label{P:UEPkernel_ideal}
     Let $\cH$ be a unitarily invariant complete Nevanlinna-Pick space on $\bB_d$
  with reproducing kernel 
  \[
  k(z,w) = \sum_{n = 0}^\infty a_n \langle z,w \rangle^n.
  \]
Let $\cI \subset \bC[z_1,\ldots,z_d]$ be a proper homogeneous ideal
and let $U=(U_1,\ldots,U_d)$ be a spherical unitary on some Hilbert space $\cE$ whose
joint spectrum is contained in $V(\cI) \cap \partial \bB_d$.
Then there exists a unique unital completely contractive homomorphism
$\rho: A(\cH_\cI) \to B(\cE)$ such that $\rho(S)=U$. If $\sum_{n=0}^\infty a_n = \infty$, then $\rho$ has the unique extension property.
\end{proposition}

\begin{proof}
  This follows as in the proof of Proposition \ref{P:UEP_bigspace} by applying
  Theorem \ref{T:rep_A_I} instead of Theorem \ref{T:A_H_representation}.
\end{proof}

Before we adapt Theorem \ref{T:bdryan}, Theorem \ref{T:hyperrigid} and Corollary \ref{C:envelope}, we need some more preparation. The following fact is standard, but we provide a proof for the sake of completeness.

\begin{lemma}
  \label{lem:kernelparts}
  Let $\cH$ be a unitarily invariant reproducing kernel Hilbert space on $\bB_d$ with kernel
  \begin{equation*}
    k(z,w) = \sum_{n=0}^\infty a_n \langle z,w \rangle^n
  \end{equation*}
  and let $\cI \subset \bC[z_1,\ldots,z_d]$
  be a homogeneous ideal. If $w \in V(\cI)$, then $\langle \cdot,w \rangle^n \in \cH_{\cI}$
  and
  \begin{equation*}
    ||\langle \cdot,w \rangle^n||_{\cH_{\cI}}^2 = \frac{1}{a_n} ||w||^{2 n}
  \end{equation*}
  for all $n \in \bN$.
\end{lemma}

\begin{proof}
  Since $\cH$ is unitarily invariant, homogeneous polynomials
  of different degree are orthogonal, and the homogeneous expansion of a function
  in $\cH$ converges in the norm of $\cH$. In particular, this is true
  for kernel functions, so if $p \in \cI$ is a homogeneous polynomial of degree $n$ and $w \in V(\cI)$
  with $||w|| < 1$, then
  \begin{equation*}
    0 = p(w) = \langle p, k(\cdot,w) \rangle = \sum_{m=0}^\infty \langle p, a_m \langle \cdot,w \rangle^m \rangle = a_n \langle p, \langle \cdot,w \rangle^n \rangle.
  \end{equation*}
  Since $\cI$ is homogeneous, it follows that $\langle \cdot,w \rangle^n$ is orthogonal to $\cI$ and thus belongs to $\cH_\cI$. Moreover,
  \begin{equation*}
    a_n || \langle \cdot, w \rangle^n ||^2 =
    \langle \langle \cdot,w \rangle^n, a_n \langle \cdot,w \rangle^n \rangle
    = \langle \langle \cdot,w \rangle^n, k(\cdot,w) \rangle = ||w||^{2 n}.
  \end{equation*}
  The case of general $w \in V(\cI)$ follows by scaling $w$.
\end{proof}

In Theorem \ref{T:bdryan}, we had to exclude the case where $\cH = H^2(\bD)$. To handle this exception
in the present setting, we require the following well-known observation.
\begin{remark}
  \label{rem:identification}
  If $Y \subset \bC^d$
  is a subspace of dimension $d' \le d$ and if $V: \bC^{d'} \to Y$
  is a unitary operator, then
  the composition operator
  \begin{equation*}
    U: H^2_d \big|_{Y \cap \bB_d} \to H^2_{d'}, \quad f \mapsto f \circ V,
  \end{equation*}
  is unitary, as
  \begin{equation*}
    \frac{1}{1 - \langle V z, V w \rangle_{\bC^d}} = \frac{1}{1 - \langle z,w \rangle_{\bC^{d'}}}
  \end{equation*}
  for all $z,w \in \bB_{d'}$. Moreover, the map
  \begin{equation*}
    \Mult( H^2_d \big|_{Y \cap \bB_d}) \to \Mult(H^2_{d'}), \quad 
    M_\varphi \mapsto U M_\varphi U^* = M_{\varphi \circ V},
  \end{equation*}
  is a completely isometric isomorphism which maps $A( H^2_d \big|_{Y \cap \bB_d})$
  onto $A(H^2_{d'})$. In particular, if $V$ is a one-dimensional subspace, then
  $H^2_d \big|_{V \cap \bB_d}$ can be identified with the Hardy space $H^2(\bD)$ on the unit disc
  in this way.
  \qed
\end{remark}

Because of this observation, we say that
$\cH_\cI$ \emph{is the Hardy space} if $\cH = H^2_d$ and $\cI$
is a radical homogeneous ideal with the property that
$V(\cI)$ is a one-dimensional subspace (recall that if $\cI$ is radical, then $\cH_{\cI}$
can be identified with $\cH \big|_{V(\cI) \cap \bB_d}$).

We can now generalize Theorem \ref{T:bdryan} and Corollary \ref{C:envelope}. Recall that under the conditions of Theorem \ref{T:SESideal}, every $\zeta\in V(\cI) \cap \partial \bB_d$ gives rise to a character
$\rho_\zeta$ on $C^*(A(\cH_\cI))$
of evaluation at $\zeta$ which annihilates the compact operators.

\begin{theorem}\label{T:bdry_ideal}
Let $\cH$ be a regular unitarily invariant complete Nevanlinna-Pick space on $\bB_d$ with reproducing kernel 
\[
k(z,w)=\sum_n a_n \langle z,w \rangle^n.
\]
Let $\cI \subset \bC[z_1,\ldots,z_d]$
be a proper homogeneous ideal which satisfies Arveson's conjecture. Suppose that
$\cH_\cI$ is not the Hardy space.
\begin{enumerate}[label=\normalfont{(\alph*)}]
    \item If $\sum_{n=0}^\infty a_n = \infty$,
then the boundary representations for $A(\cH_\cI)$ are precisely the characters $\rho_\zeta$ for $\zeta \in V(\cI)\cap \partial \bB_d $
      and the identity representation of $C^*(A(\cH_\cI))$.
    \item If 
      $\sum_{n=0}^\infty a_n < \infty$,
then the identity representation of $C^*(A(\cH_\cI))$
      is the only boundary representation  for $A(\cH_\cI)$.
  \end{enumerate}
  In particular, $C^*_{e}(A(\cH_\cI)) = C^*(A(\cH_\cI))$.
\end{theorem}

\begin{proof}
  By Theorem \ref{T:SESideal} and Lemma \ref{lem:rep_split}, the only possible irreducible
  representations of $C^*(A(\cH_\cI))$ are the identity
  representation and the characters corresponding to points
  in $V(\cI)\cap \partial \bB_d $ (note that this set may be empty). Proposition \ref{P:UEPkernel_ideal}
  shows that these characters are boundary representations if $\sum_{n=0}^\infty a_n = \infty$.
  Conversely, suppose that $\sum_{n=0}^\infty a_n < \infty$.
  As remarked in Subsection \ref{SS:an}, $\cH$ can be identified
  with a reproducing kernel Hilbert space on $\ol{\bB_d}$ in this case.
  Moreover, if $\zeta \in V(\cI) \cap \partial \bB_d$, then $k_{\zeta} = k(\cdot,\zeta) \in \cH_{\cI}$ by
  the reproducing property of the kernel. We see as in the proof
  of Theorem \ref{T:bdryan} that the character $\rho_\zeta \big|_{A(\cH_{\cI})}$ admits
  two extensions to a state on $C^*(A(\cH_\cI))$, namely $\rho_\zeta$ itself
  and the vector state associated to $k_{\zeta} / ||k_\zeta||$. These clearly
  differ, so the characters
  are not boundary representations in this case.
  
  It remains to prove that the identity representation is always a boundary representation,
  for which we will use Arveson's boundary theorem
  \cite[Theorem 2.1.1]{arveson1972}.
  By virtue of Theorem \ref{T:SESideal}, we have to show that the surjective $*$-homomorphism 
  \[
  \pi:C^*(A(\cH_\cI)) \to C(V(\cI)\cap \partial \bB_d), \quad
  S_j\mapsto z_j, \quad  (1\leq j \leq d)
  \]
 is not completely isometric on $A(\cH_\cI)$. In fact, we will show that if $\pi$ is
  isometric on $A(\cH_\cI)$, then $\cH_\cI$ is the Hardy space, contrary to our assumption.

  First of all, if $\pi$ is isometric on $A(\cH_\cI)$, then that algebra has no non-zero
  nilpotent elements, which is easily seen to imply that $\cI$ is a radical ideal.
  If $\cI = \langle z_1,\ldots,z_d \rangle$, then $\cH_\cI$ is one-dimensional so that $C^*(A(\cH_\cI))$ consists of compact operators
  and the quotient map $\pi$ is the zero map, which is clearly not isometric on $A(\cH_\cI)$.
  Hence $\cI \subsetneq \langle z_1,\ldots,z_d \rangle$, so that
  $V(\cI) \cap \partial \bB_d$ is not empty by Hilbert's Nullstellensatz.
  Choose thus $w \in V(\cI) \cap \partial \bB_d$ and let
  \[
  S_{w} =\sum_{j=1}^d \ol{w_j} S_j = P_{\cH_{\cI}} M_{\langle \cdot,w \rangle}
  \big|_{P_{\cH_{\cI}}}.
  \]
  It follows from
  \cite[Proposition 6.4]{hartz2015isom} and Lemma \ref{lem:kernelparts} that
  $||S_w^n||^2 = 1/a_n$ for all $n \in \bN$, whereas
  \begin{equation*}
    \max \{ |\langle z, w  \rangle^n| : z \in V(\cI) \cap \partial \bB_d \} = 1.
  \end{equation*}
  Since $\pi$ is an isometry we infer that $a_n = 1$ for all $n \in \bN$, and hence $\cH = H^2_d$.

  Next,
  we show that $V(\cI)$ is a subspace of $\bC^d$. An application of Lemma \ref{lem:kernelparts}
  shows that $\langle \cdot,w \rangle \in H^2_d \ominus \cI$
  for all $w \in V(\cI)$, hence for all $w$ in the linear
  span of $V(\cI)$. Consequently, for $w \in \operatorname{span} (V(\cI)) \cap \partial \bB_d$,
  we obtain the lower bound
  \begin{equation*}
    ||S_w||_{A(H^2_d \ominus \cI)} \ge ||S_w 1||_{H^2_d \ominus \cI} =
    || \langle \cdot,w \rangle||_{H^2_d} = 1.
  \end{equation*}
  Again, since $\pi$ is assumed to be
  isometric, this forces
  \begin{equation*}
    \max \{ |\langle z, w \rangle| : z \in V(\cI) \cap \partial \bB_d \} = 1,
  \end{equation*}
  which implies by the Cauchy-Schwarz inequality and
  homogeneity of $\cI$ that $w \in V(\cI)$. Using homogeneity of $\cI$ again,
  we see that the linear span of $V(\cI)$ is contained in $V(\cI)$,
  so that $V(\cI)$ is a subspace.  Since $\cI$ is radical, we may use Remark \ref{rem:identification} to conclude that $\cH_\cI$ is
  identified with $H^2_{d'}$ and that $A(\cH_\cI)$ is completely isometrically isomorphic to $A(H^2_{d'})$, where $d' = \dim(V(\cI))$. If $d' \ge 2$, it follows as in the proof of Theorem \ref{T:bdryan} that $\pi$ is not isometric, so that $V(\cI)$
  is a one-dimensional subspace of $\bC^d$, as asserted.
\end{proof}

Finally, we obtain a generalization of Theorem \ref{T:hyperrigid}.
We say that a homogeneous ideal $\cI \subset \bC[z_1,\ldots,z_d]$
is \emph{relevant} if its radical does not contain the maximal ideal
$\langle z_1,\ldots,z_d \rangle$ (see \cite[Chapter VII]{ZS75}). By Hilbert's Nullstellensatz,
$\cI$ is relevant if and only if $V(\cI) \cap \partial \bB_d$ is not empty.

\begin{theorem}
  \label{T:hyperrigid_ideal}
  Let $\cH$ be a regular unitarily invariant complete Nevanlinna-Pick space on $\bB_d$ with reproducing kernel
  \begin{equation*}
    k(z,w) = \sum_{n=0}^\infty a_n \langle z,w \rangle^n.
  \end{equation*}
  Let $\cI \subset \bC[z_1,\ldots,z_d]$
  be a relevant homogeneous ideal which satisfies Arveson's conjecture.
  Suppose that $\cH_\cI$ is not the Hardy space.
  \begin{enumerate}[label=\normalfont{(\alph*)}]
    \item If $\sum_{n=0}^\infty a_n = \infty$, then $A(\cH_{\cI})$ is hyperrigid.
    \item If $\sum_{n=0}^\infty a_n < \infty$, then $A(\cH_{\cI})$ is not hyperrigid.
  \end{enumerate}
\end{theorem}

\begin{proof}
  (a) The proof of Theorem \ref{T:hyperrigid} carries over to this setting. Indeed,
  Theorem \ref{T:SESideal} and Lemma \ref{lem:rep_split} show
  that every representation $\pi$ of $C^*(A(\cH_\cI))$ splits as
  a direct sum of $\pi_1 \oplus \pi_2$, where $\pi_1$ is unitarily equivalent
  to a multiple of the identity representation and $\pi_2$ annihilates
  the compact operators. The identity representation is a boundary
  representation by Theorem \ref{T:bdry_ideal}, so it remains
  to show that $\pi_2 \big|_{A(\cH_\cI)}$ has the unique extension property.
  Using Theorem \ref{T:SESideal} again, we see that $U = (\pi_2(S_1),\ldots,\pi_2(S_d))$
  is a spherical unitary whose joint spectrum is contained in $V(\cI) \cap \partial \bB_d$,
  so the assertion follows from Proposition \ref{P:UEPkernel_ideal}.

  (b) Since $\cI$ is relevant, the set $V(\cI) \cap \partial \bB_d$ is not empty, hence part (b) of Theorem \ref{T:bdry_ideal}
  shows that the point evaluation characters $\rho_\zeta \big|_{A(\cH_\cI)}$ with $ \zeta \in V(\cI) \cap \partial \bB_d$  do not have
  the unique extension property.
\end{proof}

We remark that if $\cI \subset \bC[z_1,\ldots,z_d]$ is a proper homogeneous ideal which is not
  relevant, then $\cI$ contains a power of the maximal ideal
  $\langle z_1,\ldots,z_d \rangle$, hence $\cH_\cI$ is finite dimensional. Lemma \ref{L:irredideal}
  therefore shows that $C^*(A(\cH_\cI)) = \cK(\cH_\cI) = B(\cH_\cI)$
  in this case, so $A(\cH_\cI)$ is always hyperrigid (for instance by Arveson's boundary theorem),
  regardless of the properties of the sequence $\{a_n\}_n$.

As in Section \ref{S:AH}, we see that Theorems \ref{T:bdry_ideal}
and  \ref{T:hyperrigid_ideal} combined show that the algebras $A(\cH_\cI)$ support
Arveson's hyperrigidity conjecture, provided that $\cI$ satisfies Arveson's essential normality
conjecture. If $\cH = H^2_d$, this is proved in \cite{KS2015} (see Proposition 4.11 and Theorem 4.12
in that paper). While the authors also extend their results to other unitarily invariant spaces  \cite[Section 5]{KS2015}, their setting is rather different from ours and there is thus very little overlap with our present work.
For instance, this can be seen by comparing our standard examples of Subsection \ref{ss:exa} to their examples
in \cite[Section 5.2]{KS2015}.

More precisely, while they do not assume that their spaces have the complete Nevanlinna-Pick property,
they require that the tuple $(M_{z_1},\ldots,M_{z_d})$ is a row contraction.
It is not hard to see that the Drury-Arveson space is the only unitarily invariant
complete Nevanlinna-Pick space on $\bB_d$ for which this happens.
Indeed, if $||M_{z_1}|| \le 1$, then $a_1 \ge 1$ (for instance by Lemma \ref{lem:kernelparts}).
On the other hand, if $\cH$ is a unitarily invariant complete Nevanlinna-Pick space, then by Lemma \ref{L:NP-char} we have
\[
1-\frac{1}{k(z,w)}=\sum_{n=1}^\infty b_n \langle z,w\rangle^n
\]
with $b_n \ge 0$ for all $n \in \bN$ and $\sum_{n=1}^\infty b_n \le 1$.
It is easily verified that $a_1 = b_1$, which forces $b_1 = 1$ and $b_n = 0$ for $n \ge 2$, so that $\cH$
is the Drury-Arveson space. Thus our setting and that of \cite{KS2015} are almost complementary.

\section{Necessity of the Nevanlinna-Pick property for coextensions}
\label{sec:NP_nec}
We saw in Corollary \ref{cor:automatic_coextension} that if
$\cH$ is a unitarily invariant complete Nevanlinna-Pick space on the unit ball, then every
unital completely contractive representation of $A(\cH)$ admits a $C^*(A(\cH))$-coextension,
and we generalized this result to the algebras $A(\cH_\cI)$ in Theorem \ref{T:rep_A_I}. We also
observed in Example \ref{E:Hardy_Bergman} that this automatic coextension property fails for
the Bergman space on the unit disc. In this section, we will show that under some mild
additional assumptions, the automatic coextension property in fact characterizes those unitarily
invariant spaces which enjoy the complete Nevanlinna-Pick property.

Throughout this section, $\cH$ denotes a unitarily invariant space on $\bB_d$ whose multiplier algebra contains
the polynomials. Recall then that $A(\cH)$ is the norm closure of the polynomial multipliers. Let
\begin{equation*}
  k(z,w) = \sum_{n=0}^\infty a_n \langle z,w \rangle^n
\end{equation*}
be the reproducing kernel of $\cH$,
with $a_0 = 1$ and $a_n > 0$ for all $n \in \bN$. Then, there is a unique sequence $\{b_n\}_{n}$ of real numbers such that
\begin{equation}
  \label{eqn:a_b}
  1- \frac{1}{\sum_{n=0}^\infty a_n t^n} = \sum_{n=1}^\infty b_n t^n
\end{equation}
holds for all $t$ in a neighbourhood of $0$.

We first examine the case where $1/k$ is a polynomial, so that all but finitely many of the numbers $b_n$ are equal to $0$.
In this case, we can use the main idea of the proof of Theorem \ref{T:NPchar} to show
that the complete Nevanlinna-Pick property is necessary for Corollary \ref{cor:automatic_coextension} to hold.

\begin{proposition}
  Let $\cH$ be a unitarily invariant space on the unit ball such that the polynomials
  are multipliers. Let $k$ be the reproducing kernel of $\cH$ and assume that $1/k$ is a polynomial. If every
  unital completely contractive representation of $A(\cH)$ admits a $C^*(A(\cH))$-coextension,
  then $\cH$ is an irreducible complete Nevanlinna-Pick space.
\end{proposition}

\begin{proof}
  Let $\cH_0$ be the closed subspace of all functions in $\cH$ which vanish at $0$. Clearly,
  $\cH_0$ is invariant for all multiplication operators, hence the restriction map
  \begin{equation*}
    A(\cH) \to B(\cH_0), \quad M_\varphi \mapsto M_\varphi \big|_{\cH_0},
  \end{equation*}
  is a unital completely contractive representation of $A(\cH)$.
  By assumption, this representation admits a $C^*(A(\cH))$-coextension $\pi$.

  Given a commuting tuple of operators $T = (T_1,\ldots,T_d)$ on a Hilbert space,
  define
  \begin{equation*}
    1/k(T,T^*) = I -
    \sum_{n=1}^\infty b_n \sum_{|\alpha| = n} \binom{n}{\alpha} T^\alpha (T^*)^\alpha.
  \end{equation*}
  Observe that since $1/k$ is polynomial, the sum is in fact finite. Moreover, Equation
  \eqref{eqn:a_b} holds for all $t \in \bD$, so that
  arguing as in the proof of Lemma \ref{L:1/k},
  we see that $1/k(M_z,M_z^*) \ge 0$, hence $1/k(\pi(M_z),\pi(M_z)^*) \ge 0$.
  Let $T = M_z \big|_{\cH_0}$. Since $T$ is
  the compression of $\pi(M_{z})$ to the
  coinvariant subspace $\cH_0$, it follows that $1/k(T,T^*) \ge 0$.

  Let $k_0 = k -1$ which is the reproducing kernel of $\cH_0$. Then 
  \[
  T_j^* k_0(\cdot,w) = \ol{w_j} k_0(\cdot,w) \quad (j=1,\ldots, d).
  \]
  Applying the inequality $1/k(T,T^*) \ge 0$ to arbitrary finite linear combinations of the kernel functions $k_0(\cdot,w), w\in \bB_d$,
  we see that the function
  \begin{equation*}
    1 - \frac{1}{k} = \frac{k_0}{k}
  \end{equation*}
  is positive semi-definite on $\bB_d$. It is well known that this implies that $b_n\geq 0$ for each $n\geq 1$ (see for instance \cite[Corollary 6.3]{hartz2015isom})   so that $\cH$ is an irreducible complete Nevanlinna-Pick space by Lemma \ref{L:NP-char}.
  \end{proof}

In the preceding proof, the observation that $1/k(M_z,M_z^*) \ge 0$ played
a crucial role.
If $1/k$ is not a polynomial, then it is not as easy to define $1/k(M_z,M_z^*)$.
To overcome this difficulty, we instead consider compressions $S^{\cI}$ of $M_z$
to the orthogonal complement of homogeneous ideals $\cI$ of finite codimension. These compressions are nilpotent,
so that the definition of $1/k(S^{\cI},(S^\cI)^*)$ is once again straightforward.

To be precise, suppose that $T = (T_1,\ldots,T_d)$ is a commuting tuple of nilpotent operators. Then we define
\begin{equation*}
  1/k(T,T^*) = I 
   - \sum_{n=1}^\infty b_n \sum_{|\alpha|=n} \binom{n}{\alpha} T^\alpha (T^*)^\alpha.
\end{equation*}
Observe that since each $T_j$ is nilpotent, the sum is in fact finite. The next
lemma is the main technical tool of this section. 

\begin{lemma}
  \label{lem:1/k_coext_NP_char}
  Let $N \in \bN$, let
  \begin{equation*}
    \cI = \langle z_1,\ldots,z_d \rangle^{N+1} \subset \bC[z_1,\ldots,z_d]
  \end{equation*}
  and let 
  \begin{equation*}
    (\cH_\cI)_0 = \{ f \in \cH_\cI: f(0) = 0 \}.
  \end{equation*}
  Put $S=S^{\cI}$ and $T = S \big|_{(\cH_{I})_0}$. Then $\cH_\cI$ consists of all polynomials
  of degree at most $N$. Moreover,
  \begin{enumerate}[label=\normalfont{(\alph*)}]
    \item $1/k(S,S^*) \ge 0$, and
    \item $1/k(T,T^*) p = \frac{b_m}{a_m} p$ for every homogeneous polynomial $p$ of degree $1 \le m \le N$.
  \end{enumerate}
\end{lemma}

\begin{proof}
  Since homogeneous polynomials of different degree are orthogonal, it follows that
  $\cH_\cI$ consists of all polynomials of degree at most $N$.
  We first show that if $n \ge 1$ and if $p$ is a homogeneous polynomial of degree $m \ge n$, then
  \begin{equation}
    \label{eqn:technical_1}
    \sum_{|\alpha| = n} \binom{n}{\alpha} M_z^\alpha (M_{z}^*)^\alpha p  = \frac{a_{m-n}}{a_m} p.
  \end{equation}
  To this end, recall that the space of homogeneous polynomials of degree $m$
  is spanned by the polynomials
  $\langle z,w \rangle^m \in \bC[z_1,\ldots,z_d]$, where $w \in \bB_d$. Moreover,
  using the fact that $M_{z_j}^*$ maps homogeneous polynomials of degree $m$ to homogeneous
  polynomials of degree $m-1$ and that $M_{z_j}^* k(\cdot,w) = \ol{w_j} k(\cdot,w)$, one readily
  sees that
  \begin{equation*}
    M_{z_j}^* \langle z,w \rangle^m = \frac{a_{m-1}}{a_m} \ol{w_j} \langle z,w \rangle^{m-1}
  \end{equation*}
  for each $j=1,\ldots, d$.
  Consequently, if $|\alpha|=n$, then
  \begin{equation*}
    (M_{z}^\alpha)^* \langle z , w \rangle^m = \frac{a_{m-n}}{a_m} \ol{w}^\alpha \langle z,w  \rangle^{m-n},
  \end{equation*}
  so that
  \begin{equation*}
    \sum_{|\alpha| = n} \binom{n}{\alpha} M_z^\alpha (M_{z}^*)^\alpha \langle z,w \rangle^{m}
    = \frac{a_{m-n}}{a_m}
    \sum_{|\alpha| = n} \binom{n}{\alpha} z^\alpha \ol{w}^\alpha \langle z,w \rangle^{m-n}
    = \frac{a_{m-n}}{a_m} \langle z,w \rangle^m,
  \end{equation*}
  which proves the claim.

  Observe further that multiplying Equation \eqref{eqn:a_b} by $\sum_{n=0}^\infty a_n t^n$
  and comparing coefficients yields that
  \begin{equation}
    \label{eqn:a_b_recursion}
    \sum_{n=1}^m b_n a_{m-n} = a_m
  \end{equation}
  for $m \ge 1$.

  With these two observations in hand, the remainder of the proof is now a straightforward computation. Indeed,
  to prove (a), recall that $\cH_{\cI}$ consists of all
  polynomials of degree at most $N$. If $p$ is a homogeneous polynomial of degree $m \le N$, then
  \begin{equation*}
    S^\alpha (S^*)^\alpha p =
    \begin{cases}
      M_z^\alpha (M_z^*)^\alpha p  & \text{ if } m \ge  |\alpha|, \\
      0  & \text{ if } m < |\alpha|,
    \end{cases}
  \end{equation*}
  so that if $1 \le m \le N$, then
  \begin{equation*}
    \sum_{n=1}^N b_n \sum_{|\alpha|=n} \binom{n}{\alpha} S^\alpha (S^*)^\alpha p
    = \sum_{n=1}^m b_n \sum_{|\alpha| = n} \binom{n}{\alpha} M_z^\alpha (M_z^*)^\alpha p
    = \sum_{n=1}^m b_n \frac{a_{m-n}}{{a_m}} p = p,
  \end{equation*}
  by Equations \eqref{eqn:technical_1} and \eqref{eqn:a_b_recursion}.
  If the degree $m$ of $p$ is equal to $0$, then the left-hand side
  is equal to $0$, so that
  \begin{equation*}
    1/k(S,S^*) = I - \sum_{n=1}^N b_n \sum_{|\alpha| = n} \binom{n}{\alpha} S^\alpha (S^*)^\alpha
  \end{equation*}
  is the orthogonal projection onto the one-dimensional space $\bC 1$. In particular,
  $1/k(S,S^*) \ge 0$.

  To prove (b), observe that $(\cH_{I})_0$ is spanned by the homogeneous polynomials of degree
  $1 \le m \le N$. Moreover, if $p$ is a homogeneous polynomial of degree $1\leq m\leq N$, then
  \begin{equation*}
    T^\alpha (T^*)^\alpha p = S^\alpha P_{(\cH_\cI)_0} (S^{*})^\alpha p=
    \begin{cases}
      M_z^\alpha (M_z^*)^\alpha p  & \text{ if } m \ge |\alpha| + 1, \\
      0  & \text{ if } m < |\alpha| + 1.
    \end{cases}
  \end{equation*}
  Consequently,
  \begin{align*}
    ( I-1/k(T,T^*))p&=\sum_{n=1}^N b_n \sum_{|\alpha| = n} \binom{n}{\alpha} T^\alpha (T^*)^\alpha p\\
    &= \sum_{n=1}^{m-1} b_n \sum_{|\alpha|=n} \binom{n}{\alpha} M_z^\alpha (M_{z}^*)^\alpha p \\
    &= \sum_{n=1}^{m-1} b_n \frac{a_{m-n}}{a_m} p
    = \frac{a_m- b_m}{a_m} p,
  \end{align*}
  where we have again used Equations \eqref{eqn:technical_1} and \eqref{eqn:a_b_recursion}. This identity proves (b).
\end{proof}

With these preliminary observations in hand, we can now show that the last statement of Theorem \ref{T:rep_A_I} characterizes
those unitarily invariant spaces which have the complete Nevanlinna-Pick property.

\begin{proposition}
  \label{prop:coext_char_ideal}
  Let $\cH$ be a unitarily invariant space on the unit ball such that the polynomials
  are multipliers. Suppose that for every proper homogeneous ideal
  of polynomials $\cI$, every
  unital completely contractive representation of $A(\cH_{\cI})$ admits a $C^*(A(\cH_{\cI}))$-coextension.
  Then $\cH$ is an irreducible complete Nevanlinna-Pick space.
\end{proposition}

\begin{proof}
  We apply the assumption to the ideals $\cI = \langle z_1,\ldots,z_d \rangle^{N+1}$ for $N \in \bN$.
  Let $S = S^{\cI}$, let $(\cH_\cI)_0$ be the space of all functions
  in $\cH_\cI$ which vanish at $0$ and let $T = S \big|_{(\cH_\cI)_0}$.   Part (a) of Lemma \ref{lem:1/k_coext_NP_char} shows that $1/k(S,S^*) \ge 0$. Now, by assumption, the unital completely contractive representation
  \begin{equation*}
    A(\cH_\cI) \to B( (\cH_\cI)_0), \quad M_\varphi \mapsto M_\varphi \big|_{(\cH_\cI)_0},
  \end{equation*}
  admits a $C^*(A(\cH_\cI))$-coextension. Hence we must also have $1/k(T,T^*) \ge 0$. By part (b) of Lemma \ref{lem:1/k_coext_NP_char}, this forces
  $b_n \ge 0$ for $1 \le n \le N$, since $a_n > 0$. Since $N\in \bN$ was arbitrary, we conclude that $\cH$ is an irreducible complete Nevanlinna-Pick space
  by Lemma \ref{L:NP-char}.
\end{proof}

If we assume in addition that $\cH$ is regular, then it suffices to consider $A(\cH)$ itself.
Recall that for a regular unitarily invariant space, the polynomials are automatically multipliers.

\begin{theorem}
  \label{thm:coext_char_regular}
  Let $\cH$ be a regular unitarily invariant space on $\bB_d$. If every unital completely
  contractive representation of $A(\cH)$ admits a $C^*(A(\cH))$-coextension, then
  $\cH$ is an irreducible complete Nevanlinna-Pick space.
\end{theorem}

\begin{proof}
Let $N \in \bN$ and let $\cI = \langle z_1,\ldots,z_d \rangle^{N+1} \subset \bC[z_1,\ldots,z_d]$. 
As before, let $S = S^{\cI}$, let $(\cH_\cI)_0$ be the space of all functions
  in $\cH_\cI$ which vanish at $0$ and let 
  \[
  T = P_{(\cH_\cI)_0}M_z \big|_{(\cH_\cI)_0}.
  \]
  Since $(\cH_{\cI})_0$ is semi-invariant for $A(\cH)$, the map
    \begin{equation*}
      A(\cH) \to B((\cH_\cI)_0), \quad M_\varphi \mapsto P_{(\cH_\cI)_0} M_\varphi \big|_{(\cH_\cI)_0},
  \end{equation*}
  is a unital completely contractive representation, and thus it admits a $C^*(A(\cH))$-coextension $\pi$. Since $\cH$ is regular, Theorem \ref{T:essentiallynormalan}
  and Lemma \ref{lem:rep_split} show that
  $\pi = \pi_1 \oplus \pi_2$, where $\pi_1$ is unitarily equivalent to
  a multiple of the identity representation and $\pi_2(M_z)$ is a spherical unitary.
  Consequently, there exist Hilbert spaces $\cE$ and $ \cL$, a spherical unitary $U=(U_1,\ldots, U_d)$ on  $\cL$ and an isometry
  \begin{equation*}
    V = \begin{bmatrix} V_1 \\ V_2 \end{bmatrix}: (\cH_\cI)_0\to (\cH \otimes \cE) \oplus \cL
  \end{equation*}
such that
  \begin{equation*}
    V^*  \big[ (M_{z_j} \otimes I_{\cE}) \oplus U_j \big] = T_j V^*,
  \end{equation*}
  and hence
  \begin{align*}
    V_1^* (M_{z_j} \otimes I_{\cE}) = T_j V_1^* \quad \text{ and } \quad
    V_2^* U_j = T_j V_2^*
  \end{align*}
  for $j=1,\ldots,d$. Since $T^\alpha = 0$ if $|\alpha| = N+1$, we deduce that
  \begin{equation*}
    V_2^* V_2 =
    V_2^* \sum_{|\alpha|={N+1}} \binom{N+1}{\alpha} U^\alpha (U^*)^\alpha V_2
    = \sum_{|\alpha|=N+1} \binom{N+1}{\alpha} T^\alpha V_2^* V_2 (T^\alpha)^* = 0,
  \end{equation*}
  hence $V_2 = 0$.
  Furthermore, if $p \in \cI$ then $p(T) = 0$, so arguing as in the proof of Theorem \ref{T:AEM_ideal},
  it follows  that the range of $V_1$ is contained in $\cH_\cI \otimes \cE$. Consequently, $V$ maps
  $(\cH_\cI)_0$ into $(\cH_{\cI}) \otimes \cE$, and $T$ coextends to $S^{\cI} \otimes I_{\cE}$.
  As in the proof of Proposition \ref{prop:coext_char_ideal}, an
  application of Lemma \ref{lem:1/k_coext_NP_char} now shows that $\cH$ is an irreducible
  complete Nevanlinna-Pick space.
\end{proof}

\bibliographystyle{amsplain}
\bibliography{literature}

\end{document}